\documentclass[11pt]{amsart}

\usepackage{amsmath,amssymb,latexsym,soul,cite,mathrsfs}

\usepackage{color,enumitem,graphicx}
\usepackage[colorlinks=true,urlcolor=blue,
citecolor=red,linkcolor=blue,linktocpage,pdfpagelabels,
bookmarksnumbered,bookmarksopen]{hyperref}
\usepackage[english]{babel}

\usepackage[left=2.9cm,right=2.9cm,top=2.8cm,bottom=2.8cm]{geometry}
\usepackage[hyperpageref]{backref}

\usepackage[colorinlistoftodos]{todonotes}
\makeatletter
\providecommand\@dotsep{5}
\def\listtodoname{List of Todos}
\def\listoftodos{\@starttoc{tdo}\listtodoname}
\makeatother

\numberwithin{equation}{section}
\newtheorem{theorem}{Theorem}[section]
\newtheorem{proposition}[theorem]{Proposition}
\newtheorem{lemma}[theorem]{Lemma}
\newtheorem{corollary}[theorem]{Corollary}

\newtheorem{claim}[theorem]{Claim}

\newtheorem{definition}[theorem]{Definition}

\newcommand\R{\mathbb R}

\begin{document}
	
	\title[EXISTENCE OF SOLUTION FOR a class of indefinite variational]
	{Existence of solution for a class of indefinite variational problems  with discontinuous nonlinearity}

	\author{Claudianor O. Alves}
	\author{Geovany F. Patricio}

	\address[Claudianor O. Alves]{\newline\indent Unidade Acad\^emica de Matem\'atica
		\newline\indent 
		Universidade Federal de Campina Grande,
		\newline\indent
		58429-970, Campina Grande - PB - Brazil}
	\email{\href{mailto:coalves@mat.ufcg.edu.br}{coalves@mat.ufcg.edu.br}}
	
	\address[Geovany F. Patricio]
	{\newline\indent Unidade Acad\^emica de Matem\'atica
		\newline\indent 
		Universidade Federal de Campina Grande,
		\newline\indent
		58429-970, Campina Grande - PB - Brazil}
	\email{\href{fernandes.geovany@yahoo.com.br}{fernandes.geovany@yahoo.com.br}}

	\pretolerance10000
	
	%\begin{document}
	
	\begin{abstract}
		\noindent This paper concerns the existence of a nontrivial solution for the following problem
		\begin{equation}
		\left\{\begin{aligned}
		-\Delta u + V(x)u & \in \partial_u F(x,u)\;\;\mbox{a.e. in}\;\;\mathbb{R}^{N},\nonumber \\
		u \in H^{1}(\mathbb{R}^{N}). 
		\end{aligned}
		\right.\leqno{(P)}
		\end{equation} 
where $F(x,t)=\int_{0}^{t}f(x,s)\,ds$, $f$ is a $\mathbb{Z}^{N}$-periodic Caratheodory function and $\lambda=0$ does not belong to the spectrum of $-\Delta+V$. Here, $\partial_t F$ denotes the generalized gradient 
of $F$ with respect to variable $t$.	
\end{abstract}

	%\thanks{Claudianor Alves was partially supported by CNPq/Brazil Proc. 304036/2013-7 ; Giovany M. Figueiredo was partially
	%supported by  CNPq, Brazil; Gaetano Siciliano  was partially supported by
	%Fapesp and CNPq, Brazil. }

	\maketitle

	%------------------------------------------------------------------------------
	\section{Introduction}
	%------------------------------------------------------------------------------
	
	At the last years a special attention has been given to indefinite problem of the type  
$$
\left\{\begin{aligned}
-\Delta u + V(x)u &= f(x,u)\;\;\mbox{in}\;\;\mathbb{R}^{N} \\
u \in H^{1}(\mathbb{R}^{N}),
\end{aligned}
\right. \leqno{(P_1)}
$$
where $ N \geq 2$, $f$ and $V$ are continuous functions, periodic with respect to $x$-variable with $f$ satisfying some technical conditions and 
$$
0 \notin \sigma(-\Delta+V),\;\;\mbox{the spectrum of}\;\; -\Delta +V. \leqno{(V)}
$$  
Hereafter, we understand that $(P)$ is  indefinite when $u=0$ is not a local minimum for the corresponding energy functional. Probably the first articles to consider condition $(V)$ were
Alama and Li \cite{AmLi}, Angenent \cite{Angenent}, and Coti Zelati and Rabinowitz \cite{Coti}. 

Of particular interest in our work is the article by  Kryszewski and Szulkin \cite{Kryszewski} where the authors have studied the existence of solution for $(P_1)$ by assuming the following conditions on $f$: 
\begin{equation}\label{03}
|f(x,t)| \leq c(|t|+ |t|^{p-1}), \;\; \forall\; t \in \mathbb{R}\;\;\mbox{and}\;\; x \in \mathbb{R}^{N},
\end{equation}
and
\begin{equation}\label{04}
0< \theta F(x,t) \leq t f(x,t)\;\;\forall\; t \in \mathbb{R},\;\;F(x,t)=\int_{0}^{t} f(x,s) ds
\end{equation}
for some $c>0$, $\theta>2$ and $2<p<2^{*}$ where $2^{*}=\frac{2N}{N-2}$ if $N\geq 3$ and $2^{*}=+\infty$ if $N=2$. The above hypotheses guarantee that the energy functional associated with $(P_1)$ given by
\begin{equation} \label{varphi1}
\varphi(u)=\frac{1}{2} \int_{\mathbb{R}^{N}} (|\nabla u|^{2}+V(x)|u|^{2})dx- \int_{\mathbb{R}^{N}} F(x,u) dx,\; u \in  H^{1}(\mathbb{R}^{N}),
\end{equation}
is well defined and belongs to $C^{1}(H^{1}(\mathbb{R}^{N}), \mathbb{R})$. By $(V)$, there is an equivalent inner product $\langle \, , \,  \rangle$ in $H^{1}(\mathbb{R}^{N})$ such that
\begin{equation} \label{varphi2}
\varphi(u)= \frac{1}{2}||u^{+}||^{2}-\frac{1}{2}||u^{-}||^{2}- \int_{\mathbb{R}^{N}} F(x,u) dx,
\end{equation}
where $||u||=\sqrt{\left<u,u\right>}$ and $H^{1}(\mathbb{R}^{N})= E^{+}\oplus E^{-}$ corresponds to the spectral decomposition of $-\Delta+ V$ with respect to the positive and negative part of the spectrum with $u=u^{+}+u^{-}$, where $u^{+} \in E^{+}$ and $u^{-} \in E^{-}$. In order to show the existence of solution for $(P_1)$,  Kryszewski and Szulkin introduced a new and interesting generalized link theorem. Since this linking theorem works very well for a large class of indefinite problems, many authors have used it to study related problems, see for example, Arioli and Szulkin \cite{Arioli}, Bartsch and Ding \cite{Thomas}, Chabrowski and Szulkin \cite{Chabrowski}, Szulkin and Zou \cite{A. Szulkin and Zou}, Schechter and Zou \cite{Schechter}, Ackermann \cite{Ackermann}, do \'O and Ruf  \cite{do O}, Furtado and Marchi \cite{Furtado} and Tang \cite{X.H2, X.H, X.H1}.

In \cite{G.B Li}, Li and Szulkin have improved the generalized link theorem obtained in \cite{Kryszewski} to establish the existence of solution for a class of indefinite problem with $f$ being asymptotically linear at infinity.  Chen and Wang \cite{Chen} also showed a new infinite-dimensional linking theorem, which was also inspired by  \cite{Kryszewski}, to study a new class of indefinite problem.

The existence of solution for $(P_1)$ has been obtained of a different way in Pankov \cite{Pankov1}, Pankov and Pfl\"uger \cite{Pankov2} and Szulkin and Weth \cite{Szulkin}. In   \cite{Pankov1} and \cite{Pankov2}, the existence of ground state solution was established by supposing $f \in C^{1}(\mathbb{R}^N,\mathbb{R})$ and that there is $\theta \in (0,1)$ such that
$$
0<t^{-1}f(x,t)\leq \theta f'_t(x,t), \quad \forall t \not=0 \quad \mbox{and} \quad x \in \mathbb{R}^N. \eqno({h_3})
$$
In  \cite{Pankov1}, Pankov found a ground state solution by minimizing the energy functional  $J$ on the set
$$
\mathcal{O}=\left\{u\in H^{1}(\R^{N})\setminus E^{-}\ ;\ J'(u)u=0\text{ and }J'(u)v=0,\forall\ v\in E^{-}\right\}. 
$$ 
The reader is invited to see that if $E^{-}=\{0\}$,  the set $\mathcal{O}$ is exactly the Nehari manifold associated with $J$. Hereafter, we say that  $u_0 \in H^{1}(\mathbb{R}^{N})$ is a {\it ground state solution} if 
$$
J'(u_0)=0 \quad \mbox{and} \quad J(u_0)=\inf_{w \in \mathcal{O}}J(w).
$$

In \cite{Szulkin}, Szulkin and Weth established the existence of ground state solution for problem $(P_1)$ by completing the study made in  \cite{Pankov1}, in the sense that, they also minimize the energy functional on  $\mathcal{O}$, however they have used more weaker conditions on $f$, for example $f$ is continuous, $\mathbb{Z}^N$-periodic in $x$ and satisfies 
$$
|f(x,t)|\leq C(1+|t|^{p-1}), \;\; \forall t \in \mathbb{R} \quad \mbox{and} \quad x \in \mathbb{R}^N  \eqno{(h_4)}
$$ 
for some $C>0$ and $p \in (2,2^{*})$.
$$
f(x,t)=o(t) \,\,\, \mbox{uniformly in } \,\, x \,\, \mbox{as} \,\, |t| \to 0. \eqno{(h_5)}
$$
$$
F(x,t)/|t|^{2} \to +\infty \,\,\, \mbox{uniformly in } \,\, x \,\, \mbox{as} \,\, |t| \to +\infty, \eqno{(h_6)}
$$
and
$$
t\mapsto f(x,t)/|t| \,\,\, \mbox{is strictly increasing on} \,\,\, \mathbb{R} \setminus \{0\}. \eqno{(h_7)}
$$
The same approach was used by Alves and Germano \cite{AG,AG2},  and Zhang, Xu and Zhang \cite{ZXZ, ZXZ2}.

For the reader interested in indefinite problem there is a rich literature and we would like to cite the papers by Ding and Shixia \cite{Ding Y},  Qianqiao and Jaroslaw \cite{Guo}, Liu \cite{SLiu}, Xiaoyan and Xianhua \cite{Tang X}, Hui, Tang \cite{X.H2, X.H, X.H1},  Hui, Xub and Fubao \cite{Z. Hui} and their references.

After a bibliography review we have observed that there is no paper involving the problem $(P_1)$ with $f$ being a discontinuous function, that is, there is no paper that consider the existence of solution for problem like 
		\begin{equation}
\left\{\begin{aligned}
-\Delta u + V(x)u & \in \partial_u F(x,u)\;\;\mbox{a.e. in}\;\;\mathbb{R}^{N},\nonumber \\
u \in H^{1}(\mathbb{R}^{N}). 
\end{aligned}
\right.\leqno{(P)}
\end{equation}  
This type of problem becomes interesting, at least of the mathematics point of view, because we cannot use the classical variational methods, since the energy functional is only locally Lipschitz. Here, we intend to prove the same result found in  Kryszewski and Szulkin \cite{Kryszewski} by allowing $f$ to be a discontinuous function. One of the main difficulty in the present paper was to prove a version of the abstract framework developed in  \cite{Kryszewski} for locally Lipschitz, since in that paper the authors worked with  $C^{1}$-functionals, for more details see Section 5.

The interest in the study of nonlinear partial differential equations with discontinuous nonlinearities has increased because many free boundary problems arising in mathematical physics may be stated in this form. Among these problems, we have the seepage surface problem and the Elenbaas equation, see for example \cite{Chang1,chang2, chang3}.

A rich literature is available for problems with discontinuous nonlinearities, and we refer the reader to Chang \cite{Chang1}, Ambrosetti and Badiale \cite{Badiale}, Cerami \cite{cerami}, Alves et al. \cite{alves1}, Alves et al. \cite{alves2}, Alves and Bertone \cite{alves3}, Alves and Nascimento \cite{alves4}, Badiale \cite{badiale2}, Dinu \cite{dinu} and their references. Several techniques have been developed or applied in their study, such as variational methods for nondifferentiable functionals, lower and upper solutions, global branching, and the theory of multivalued mappings.

Motivated by ideas found in \cite{Kryszewski}, we study the existence of nontrivial solution for problem $(P)$ by supposing the following conditions on function $f$:
\begin{itemize}
	\item [($f_{1}$)] $f:\mathbb{R}^{N}\times \mathbb{R}\rightarrow \mathbb{R}$ is a Caratheodory function, $\mathbb{Z}^{N}$-periodic with respect to variable $x$ and the functions 	
	$$
	\underline{f}(x,t)= \lim_{r\downarrow 0} ess \inf\{f(x,s); |s-t|<r\}
	$$
	and
	$$
	\overline{f}(x,t)= \lim_{r\downarrow 0} ess \sup\{f(x,s); |s-t|<r\}
	$$
	are $N$-measurable functions, see Chang \cite{Chang1,chang3} for more details. 
	
	\item [($f_{2}$)] There exists $c>0$ such that 
	$$|f(x,t)|\leq c(1+|t|^{p-1}),\;\forall\;t \in \mathbb{R},\;\forall\;x\in \mathbb{R}^{N},$$
	with $p>2$ if $N=1,2$ and $p\in (2,2^{*})$ if $N\geq 3$.

	\item  [($f_{3}$)] $f(x,t)=o(|t|)$ 	uniformly with respect to $x \in \mathbb{R}^{N}$ as $|t|\rightarrow 0$.
	
	\item [($f_{4}$)] There exists $\theta>2$ such that
	$$0<\theta F(x,t)\leq \min\{t f(x,t),t\xi\}, \;\;\;\forall\; t \in \mathbb{R} \quad \mbox{and} \quad \xi \in \partial_t F(x,t) ,$$
	uniformly with respect to $x \in \mathbb{R}^{N}$, where 
	$$F(x,t)=\int_{0}^{t}f(x,s)ds.$$

\end{itemize}

\vspace{0.5 cm}

Our main result is the following:
\begin{theorem} \label{Teorema1}
	Assume $(V)$ and $(f_{1})-(f_{4})$. Then, the problem $(P)$ has a nontrivial solution.
\end{theorem}

Before concluding this introduction, we would like to say that an example of a function $f$ that satisfies the conditions above is the following
$$
f(t)=H_e(|t|-a)|t|^{p-2}t+|t|^{q-2}t, \quad \forall t \in \mathbb{R},
$$ 
with $a > 0$, $2<q<p<2^*$ where 
$$
2^*=
\left\{
\begin{array}{l}
\frac{2N}{N-2}, \quad \mbox{if} \quad N \geq 3\\
+\infty, \quad \mbox{if} \quad N=1,2,
\end{array}
\right.
$$ 
and $H_e:\mathbb{R} \to \mathbb{R}$ denotes the Heaviside function, that is, 
$$
H_e(t)=
\left\{
\begin{array}{l}
0, \quad \mbox{if} \quad t \leq 0\\
1, \quad \mbox{if} \quad t > 0.
\end{array}
\right.
$$ 

%In the sequel, we denote by $\Psi: H^{1}(\mathbb{R}^N) \to \mathbb{R}$ the functional given by
%$$
%\Psi(u)=\int_{\mathbb{R}^{N}}F(x,u).
%$$

The plan of the paper is as follows. In Section 2 we recall some definitions and basic results on the critical point theory of Locally Lipschitz Functionals. In Section 3 we show some preliminaries about Orlicz spaces. In Section 4 we study the properties of the functional $\Psi$. In Section 5 we prove a new Linking Theorem for Locally Lipschitz Functionals.  Finally, in Section 6, we employ the new Linking Theorem to establish the existence of a nontrivial solution for problem $(P)$. 

\vspace{0.5 cm}

\noindent \textbf{Notation:} From now on, otherwise mentioned, we use the following notations:
\begin{itemize}

	\item $X^{*}$ denotes the dual topological space of $X$ and $||\;\;\;||_{*}$ denotes the norm in $X^{*}$.
	
	\item $B_r(u)$ is an open ball centered at $u \in X$ with radius $r>0$.
	
%	\item   $C$ denotes any positive constant, whose value is not relevant.
	
	\item  $||\,\,\,||_p$ denotes the usual norm of the Lebesgue space $L^{p}(\mathbb{R}^N)$, for $p \in [1,+\infty]$.
	
	\item  $||\,\,\,||_{H}$ denotes the usual norm of the Orlicz space $L^{H}(\mathbb{R}^N)$ associated the $N$-function $H$.
	
	\item If $u:\mathbb{R}^N \to \mathbb{R}$ is mensurable function, the integral $\int_{\mathbb{R}^N}u\,dx$ will be denoted by $\int_{\mathbb{R}^N}u$.
	
	\item $o_{n}(1)$ denotes a real sequence with $o_{n}(1)\to 0$ as $n \to +\infty$.
	
	\item $l:X\rightarrow \mathbb{R}$ denotes a continuous linear functional.
	
\end{itemize}

\section{Basic results from nonsmooth analysis}\label{08}
	
	In this section, for the reader's convenience, we recall some definitions and basic results on the critical point theory of Locally Lipschitz Functionals as developed by Chang \cite{Chang1}, Clarke \cite{Clarke, Clarke1}, and Grossinho and Tersin \cite{rosario}.
	
	Let $(X, ||\;||)$ be a real Banach space. A functional $I :X \rightarrow \mathbb{R} $ is locally Lipschitz, $I\in Lip_{loc}(X, \mathbb{R})$ for short, if given $u \in X$ there is an open neighborhood $V:=V_{u} \subset X$ of $u$, and a constant $K=K_{u}>0$ such that
	
	\begin{equation*}
	|I(v_{2})-I(v_{1})| \leq K ||v_{1}-v_{2}||,\;\;v_{i} \in V,\;i=1,2.
	\end{equation*}
	The generalized directional derivative of $I$ at $u$ in the direction of $v \in X$ is defined by
	\begin{equation*}
	I^{\circ}(u;v)= \limsup_{h\rightarrow 0, \delta \downarrow 0}\frac{1}{\delta}\left(I(u+h+\delta v)-I(u+h)\right).
	\end{equation*}
	Hence, $I^{\circ}(u; \cdot)$ is continuous, convex and its subdifferential at $z \in X$ is defined by
	\begin{equation*}
	\partial I^{\circ}(u; z)=\{\mu \in X^{*}\;;\; I^{\circ}(u;v) \geq \left<\mu, v-z\right>\,;\forall\; v \in X \}
	\end{equation*}
where $\left<\cdot, \cdot\right>$ is the duality pairing between $X^{*}$ and $X$. The generalized gradient of $I$ at $u$ is the set
	\begin{equation*}
	\partial I(u)= \{\xi \in X^{*}\;;\; I^{\circ}(u;v) \geq \left<\xi, v\right>\,;\forall\; v \in X \}.
	\end{equation*} 
Moreover, we denote by $\lambda_{I}(u)$ the following real number 
$$
\lambda_{I}(u):=\min\{||\xi||_{X^{*}}: \xi \in \partial I(u)\}.
$$
We recall that $u \in X$ is a critical point of $I$ if $0 \in \partial I(u)$, or equivalently, when $\lambda_I(u)=0$.
\begin{lemma}
		If $I$ is continuously differentiable to Fr\'echet in an open neighborhood of $u \in X$, we have $\partial I(u)= \{I'(u)\}$.
	\end{lemma}
	
	\begin{lemma}\label{35}
		If $I \in C^{1}(X, \mathbb{R})$ and $J \in Lip_{loc}(X, \mathbb{R})$, then for each $u \in X$
		$$\partial(I+J)(u)= \partial I(u) + \partial J(u).$$
	\end{lemma}
	
	\begin{lemma}\label{19}
	Assume $\phi \in C^{1}([0,1], X)$ and $ I \in  Lip_{loc}(X, \mathbb{R})$. Then, the function $\Upsilon:= I \circ \phi:[0,1]\rightarrow \mathbb{R}$ is differentiable  a.e in $[0,1]$ and
		$$
		\Upsilon'(t) \leq \max\{\left<\xi, \phi'(t)\right>; \xi \in \partial I (\phi(t))\}	\quad \mbox{a.e. in} \quad  [0,1].
		$$
	
	\end{lemma}
	
	\section{Preliminaries about Orlicz spaces}
	
	In this section, we recall some properties involving Orlicz and Orlicz-Sobolev spaces that can be found in \cite{Adams1, Fukagai 1, RAO}. To begin with, let us say that a continuous
	function $H: \mathbb{R}\rightarrow [0, +\infty)$ is a $N$-function if:
	\begin{itemize}
		\item [(i)] $H$ is convex.
		\item [(ii)] $H(t)=0\Leftrightarrow t=0$.
		\item [(iii)] $\displaystyle \lim_{t\rightarrow 0} \frac{H(t)}{t}=0$ and  $\displaystyle\lim_{t\rightarrow \infty} \frac{H(t)}{t}=+\infty$.
		\item [(iv)] $H$ is even.
	\end{itemize}
	From now on, a $N$-function $H$ verifies the $\Delta_{2}$-condition, denoted by $H \in \Delta_{2}$, if
	\begin{equation*}
	H(2t) \leq kH(t)\;\;\forall\; t\geq 0,
	\end{equation*}
	for some constant $k>0$. 
	
	The complementary function ( or conjugate function ) $\tilde{H}$ associated with $H$ is given by the Legendre's transformation, that is,
	$$\tilde{H}= \max_{t \geq 0} \{st-H(t)\}\;\;\mbox{for}\;\; s \geq 0.$$
	The functions $H$ and $\tilde{H}$ are complementary to each other, and $\tilde{H}$ is also a N-function.
	
	In what follows, fixed an open set $\Omega \subset \mathbb{R}^{N}$ and a $N$-function $H$, we define the Orlicz space associated with $H$ as
	$$L^{H}(\Omega) = \left\{u \in L^{1}_{loc}(\Omega) \; : \; \int_{\Omega} H\left(\dfrac{|u|}{\lambda}\right) < +\infty, \quad \text{for some} \, \, \lambda >0 \right\}.$$
	The space $L^{H}(\Omega)$ is a Banach space endowed with Luxemburg norm given by
	$$
	||u||_{H} = \inf \left\{\lambda > 0 \; : \; \int_{\Omega}H\left(\dfrac{|u|}{\lambda}\right) \le 1 \right\}.
	$$
	%The convexity of $H$ implies in the inequality below
	%$$||u||_{H}\leq 1\Leftrightarrow \int_{\Omega} H(|u|) \leq 1.$$
	We would like point out that in Orlicz spaces we also have a H\"older type inequality, which is given by
	$$
	\left\lvert\int_{\Omega} u v \right\lvert \le 2||u||_{H}||v||_{\tilde{H}},\;\;\forall\; u  \in L^{H}(\Omega) \quad \mbox{and} \quad u  \in L^{\tilde{H}}(\Omega)
	$$
	The space $L^{H}(\Omega)$ is separable and reflexive when $H$ and $\tilde{H}$ satisfy the $\Delta_{2}$-condition. Moreover, $\Delta_{2}$-condition implies 
	$$
	u_{n}\rightarrow u \;\;\mbox{in}\;\; L^{H}(\Omega)\Leftrightarrow \int_{\Omega} H(|u_{n}-u|)\rightarrow 0.
	$$
	The Orlicz-Sobolev space is defined as
	\begin{eqnarray}
	W^{1, H}(\Omega)=\{u &\in& L^{H}(\Omega); \exists
	\; v_{j} \in L^{H}(\Omega)\;\mbox{with}\; \int_{\Omega} u \frac{\partial\psi}{\partial x_{j}}=-\int_{\Omega} v_{j}\psi \nonumber \\ 
	&\mbox{for}& j=1,...,N\;\mbox{and}\; \forall\; \psi \in C_{0}^{\infty}(\Omega)\} \nonumber
	\end{eqnarray}
	endowed with the norm
	$$||u||_{1, H}= ||u||_{H}+ ||\nabla u||_{H}.$$
	It is well known that $(W^{1, H}(\Omega),||u||_{1, H})$ is a Banach space.
	
	Another important inequality was proved by Donaldson and Trudinger \cite{T.K}, which establishes the existence of a constant $S_{N} = S(N) > 0$ such that
	$$||u||_{H_{*}}\leq 
	S_{N} ||u||_{1, H},\; u \in W^{1, H}(\mathbb{R}^{N}),$$
	where $H_{*}$ is the Sobolev conjugate function of $H$, defined by
	$$H_{*}^{-1}(t)= \int_{0}^{t} \frac{H^{-1}(s)}{s^{(N+1)/N}}ds,\;\;\mbox{for}\;\; t>0.$$
	The last inequality shows the embedding $W^{1, H}(\Omega)\hookrightarrow L^{H_{*}}(\Omega)$  is continuous. Furthermore, if the limits below occur  
	\begin{equation*}
	\limsup_{|t|\rightarrow 0} \frac{H(t)}{|t|^{2}}=c_{1}\;\;\mbox{and}\;\; \limsup_{|t|\rightarrow +\infty}\frac{H(t)}{{2^{*}}^{2^{*}}|t|^{2^{*}}} =0,
	\end{equation*}
the embedding $H^{1}(\mathbb{R}^{N})\hookrightarrow L^{H}(\mathbb{R}^{N})$ is also continuous.

An important N-function in our study is $\Phi(t)= |t|^{2}+|t|^{p}$ that satisfies 
\begin{equation} \label{emphi*}
2 \leq \frac{\Phi'(t)t}{\Phi(t)} \leq p,\;\;\mbox{for}\;\; t>0,
\end{equation}
	and
\begin{equation} \label{emphi}
	\limsup_{|t|\rightarrow 0} \frac{\Phi(t)}{|t|^{2}}=1\;\;\mbox{and}\;\; \limsup_{|t|\rightarrow +\infty}\frac{\Phi(t)}{{2^{*}}^{2^{*}}|t|^{2^{*}}} =0.
\end{equation}
These limits guarantee that the embedding $H^{1}(\mathbb{R}^{N})\hookrightarrow L^{\Phi}(\mathbb{R}^{N})$ is continuous. The inequality (\ref{emphi*}) is crucial in our approach, because it ensures that $\Phi$ and $\tilde{\Phi}$ satisfy $\Delta_{2}$-condition, then  $L^{\Phi}(\mathbb{R}^{N})$ and $L^{\tilde{\Phi}}(\mathbb{R}^{N})$ are reflexive spaces. 

	\section{Some properties of the functional $\Psi$}
	
	Let $F: \mathbb{R}^{N} \times \mathbb{R} \rightarrow \mathbb{R}$ be a mensurable function for each $t \in \mathbb{R}$ and Locally
	Lipschitzian for each $x \in \mathbb{R}^{N}$. By definition
	$$\partial_{t} F(x,t)= \{\mu \in \mathbb{R}\;:\; F^{\circ}(x,t;r) \geq \mu r,\; r \in \mathbb{R}\},$$
	where $F^{\circ}(x,t;r)$ denotes the generalized directional derivative of $t\mapsto F(x,t)$ in the direction
	of $r$, that is,
	\begin{equation*}
	F^{\circ}(x,t;r)= \limsup_{h\rightarrow t, \lambda \downarrow 0} \frac{F(x, h+\lambda r)-F(x,h)}{\lambda}.
	\end{equation*}
	Assumptions $(f_{1})-(f_{3})$ imply that, for any $\varepsilon>0$, there is $C_{\varepsilon}>0$ such that
	\begin{equation*}
	|f(x,t)|\leq \varepsilon |t|+C_{\varepsilon}|t|^{p-1},\;\forall\;t \in \mathbb{R},\;\forall\;x\in \mathbb{R}^{N},
	\end{equation*}
 consequently,
		\begin{equation*}
	|F(x,t)|\leq \frac{\varepsilon}{2} |t|^{2}+\frac{C_{\varepsilon}}{p}|t|^{p},\;\forall\;t \in \mathbb{R},\;\forall\;x\in \mathbb{R}^{N}.
	\end{equation*}
From this, it is easy to check that functional $\Psi: H^{1}(\mathbb{R}^{N})\rightarrow \mathbb{R}$ given by
\begin{equation} \label{Psi}
	\Psi(u)=\int_{\mathbb{R}^{N}}F(x,u),
\end{equation}
	is well defined. However, in order to apply variational methods for locally Lipschitz functionals, it is better to consider the functional $\Psi$ in a more appropriated domain, that is, \linebreak $\Psi: L^{\Phi}(\mathbb{R}^{N})\rightarrow \mathbb{R}$, for $\Phi(t)=|t|^{2}+|t|^{p}$, where $ L^{\Phi}(\mathbb{R}^{N})$ denotes the Orlicz space associated with the $N$-function $\Phi$. 
	
	Since $\Phi$ satisfies $\Delta_{2}$-condition, we can guarantee that given $J \in (L^{\Phi}(\mathbb{R}^{N}))^{*}$, then
	$$J(u)= \int_{\mathbb{R}^{N}} vu\;,\;\;\forall\; u \in L^{\Phi}(\mathbb{R}^{N}),$$
	for some $v \in L^{\tilde{\Phi}}(\mathbb{R}^{N})$, where $\tilde{\Phi}$ is the conjugate function of $\Phi$.
	
The above information involving the functional $\Psi$ are crucial in our approach, because they are used to prove that the inclusion below holds
	$$\partial \Psi (u)\subset \partial_{t} F(x,u)= [\underline{f}(x, u(x)), \overline{f}(x, u(x))]\;\;\mbox{a.e in}\;\; \mathbb{R}^{N},$$
where
	$$\underline{f}(x,t)= \lim_{r\downarrow 0} ess \inf\{f(x,s); |s-t|<r\}$$
and
	$$\overline{f}(x,t)= \lim_{r\downarrow 0} ess \sup\{f(x,s); |s-t|<r\}.$$

We would like point out that for elliptic problem involving discontinuous nonlinearity in bounded domains, the Orlicz space $L^{\Phi}(\mathbb{R}^{N})$ is not necessary, because in that case the growth of the function $f$ implies that the functional $\Psi$ is well defined in the Lebesgue space $L^{p}(\Omega)$, for more details see Chang \cite[Section 2]{Chang1}. However, in the present paper we are working in $\mathbb{R}^N$ and the conditions on $f$ yield 
$$
|F(x,t)| \leq C(|t|^{2}+|t|^{p}), \quad t \in \mathbb{R}, 
$$
then $\Psi$ is not well defined in the space $L^{p}(\mathbb{R}^N)$. The above estimate involving the function $F$ suggests that the best space to work is the Orlicz space $ L^{\Phi}(\mathbb{R}^{N})$.

Before continuing our study, we would like to mention that by  $(f_1)-(f_{3})$ the condition below occurs:
	\begin{itemize}
		\item [($F_{*}$)]There exist $C_0, C_1>0$ such that 
		$$|\xi| \leq C_0(|t| +|t|^{p-1}) \leq C_1\Phi'(|t|), \;\;\forall\; \xi \in \partial_{t} F(x,t)\;,\forall\; t \in \mathbb{R},\;\forall\; x \in \mathbb{R}^{N}.$$
	\end{itemize}

The next three results establish important properties of the functional $\Psi$ given in (\ref{Psi}).   

	\begin{lemma}[See \mbox{\cite[Lemma 3.1]{Alves1}}]
		Assume ($F_{*}$). Then, the functional $\Psi: L^{\Phi}(\mathbb{R}^{N})\rightarrow \mathbb{R}$ given by
		$$\Psi(u)= \int_{\mathbb{R}^{N}} F(x,u),\; u \in L^{\Phi}(\mathbb{R}^{N}),$$
		is well defined and $\Psi \in Lip_{loc}(L^{\Phi}(\mathbb{R}^{N}), \mathbb{R})$.
	\end{lemma}
	
	\begin{theorem}[See \mbox{\cite[ Theorem 4.1]{Alves}}]\label{72}
		Assume ($F_{*}$) and that $\underline{f}(x,t)$ and $\overline{f}(x,t)$ are $N$-mensurable
		functions. Then for each $u \in L^{\Phi}(\mathbb{R}^{N})$,
		\begin{equation}\label{67}
		\partial \Psi (u) \subset \partial_{t} F(x, u) = [\underline{f}(x,u(x)), \overline{f}(x,u(x))]\;\;\mbox{a.e in}\;\; \mathbb{R}^{N}.
		\end{equation}
	\end{theorem}
The inclusion above means that given $\xi \in  \partial \Psi (u) \subset (L^{\Phi}(\mathbb{R}^{N}))^{*}\approx L^{\tilde{\Phi}}(\mathbb{R}^{N})$, there is $\tilde{\xi} \in L^{\tilde{\Phi}}(\mathbb{R}^{N})$ such that  
	\begin{itemize}
		\item $\left<\xi, v \right>= \displaystyle \int_{\mathbb{R}^{N}} \tilde{\xi}v ,\;\forall\; v \in L^{\Phi}(\mathbb{R}^{N})$,
		\item $\tilde{\xi}(x) \in \partial_{t} F(x, u(x)) = [\underline{f}(x,u(x)), \overline{f}(x,u(x))]\;\;\mbox{a.e in}\;\; \mathbb{R}^{N}.$
	\end{itemize}

\begin{proposition}[ See  \cite{Clarke1} and \cite{Chang1}]\label{38}	
	Let $\Omega \subset \mathbb{R}^{N}$ be a domain, $(u_{n}) \subset L^{\Phi}(\Omega)$ and $(\rho_{n}) \subset L^{\tilde{\Phi}}(\Omega)$ with $\rho_{n} \in \partial \Psi(u_{n})$. If $u_{n} \rightarrow u_{0}$ in $L^{\Phi}(\Omega)$ and $\rho_{n} \stackrel{\ast}{\rightharpoonup} \rho$ in $L^{\tilde{\Phi}}(\Omega)$, then $\rho_{0} \in \partial \Psi(u_{0})$.
\end{proposition}

The next lemma is a technical result that will be used in the proof of Proposition \ref{7}.

	\begin{lemma}\label{C1}
		Suppose $(F_{*})$ and fix $u_{0}\in H^{1}(\mathbb{R}^{N})$. If $w \in L^{\Phi}(\mathbb{R}^{N})\backslash\{0\}$, then
		\begin{equation*}
		\limsup_{ \lambda \downarrow 0} \left[\int_{B_{R}^{c}}\frac{|F(x,u_{0}+\lambda w)- F(x,u_{0})|}{\lambda}\right]\leq C_{1}\int_{B_{R}^{c}} \Phi'(|u_{0}|) |w|,\;\forall\; R>0.
		\end{equation*}
		Consequently, since $\Phi'(|u_{0}|) |w| \in L^{1}(\mathbb{R}^{N})$, we get
		\begin{equation*}
		\lim_{R\rightarrow +\infty}\left(\limsup_{ \lambda \downarrow 0} \left[\int_{B_{R}^{c}}\frac{|F(x,u_{0}+\lambda w)- F(x,u_{0})|}{\lambda}\right]\right)=0.
		\end{equation*}
	\end{lemma}
	\begin{proof}
		
		For $w \in L^{\Phi}(\mathbb{R}^{N})\backslash\{0\}$, by Lebourg's Theorem, there is $\xi_{\lambda} \in \partial_{t} F(x, \theta_{\lambda})$ with $\theta_{\lambda} \in [u_{0}, u_{\lambda}]$, where $u_{\lambda}=u_{0}+\lambda w$, such that
		\begin{equation*}
		|F(x,u_{0}+\lambda w)- F(x,u_{0})|= |\left<\xi_{\lambda}, u_{0}-u_{\lambda}\right>|=\lambda|\left<\xi_{\lambda}, w \right>|.
		\end{equation*}
		Using $(F_{*})$, we have
		\begin{eqnarray}
		|F(x,u_{0}+\lambda w)- F(x,u_{0})|&\leq& \lambda |\xi_{\lambda}|\;|w| \nonumber \\
		&\leq& \lambda C_{1}\Phi'(|\theta_{\lambda}|) |w|, \nonumber
		\end{eqnarray}
		that is,
		\begin{equation*}
		\frac{|F(x,u_{0}+\lambda w)- F(x,u_{0})|}{\lambda}\leq C_{1} \Phi'(|\theta_{\lambda}|) |w|.
		\end{equation*}
		Due to the fact that $\theta_{\lambda}=tu_{0}+(1-t)u_{\lambda}$, $0<t<1$, we get
		\begin{equation*}
		|\theta_{\lambda}(x)|\leq |u_{0}(x)|+ \lambda |w(x)|:=\eta_{\lambda}(x).
		\end{equation*}
		
		Remember that we are working with the $N$-function $\Phi(t)= |t|^{2}+|t|^{p}$, $p\in (2,2^{*})$, then $\Phi'$ is increasing for $t>0$, that is, $\Phi'(|\theta
		_{\lambda}|)\leq \Phi'(\eta_{\lambda})$.
		Therefore,
		\begin{equation}\label{C3}
		\limsup_{ \lambda \downarrow 0}\int_{B_{R}^{c}} \frac{|F(x,u_{0}+\lambda w)- F(x,u_{0})|}{\lambda}\leq C_{1} \limsup_{ \lambda \downarrow 0}\int_{B_{R}^{c}}\Phi'(\theta_{\lambda}) |w|,\;\forall\; R>0,
		\end{equation}
		where $B_{R}^{c}$ denotes the complementary of $B_{R}$ in $\mathbb{R}^{N}.$
		
		Now let us show that
		\begin{equation*}
		\lim_{\lambda\rightarrow 0^{+}} \int_{B_{R}^{c}} \Phi'(\theta_{\lambda}) |w|= \int_{B_{R}^{c}} \Phi'(|u_{0}|) |w|,\;\forall\; R>0.
		\end{equation*}
		First note that
		$$\Phi'(\theta_{\lambda})|w(x)| \rightarrow \Phi'(|u_{0}|)|w(x)|\;\;\mbox{a.e in}\;\; B_{R}^{c}\;\;\mbox{as}\;\;\lambda\rightarrow 0^{+}.$$
		On the other hand, for $0<\lambda \leq 1$, we obtain
		\begin{eqnarray}
		|\Phi'(\theta_{\lambda})||w| & \leq & [2(|u_{0}|+  \lambda |w| )+ p(|u_{0}|+  \lambda |w|)^{p-1}]\; |w|\nonumber \\
		&\leq & 2|u_{0}|\;|w|+2|w|^{2}+p2^{p-1}(|u_{0}|^{p-1}\;|w|+ |w|^{p})\in L^{1}(\mathbb{R}^{N}).\nonumber
		\end{eqnarray}
		By Lebesgue's dominated convergence theorem
		\begin{equation}\label{C2}
		\lim_{\lambda\rightarrow 0^{+}} \int_{B_{R}^{c}} \Phi'(\theta_{\lambda}) |w|= \int_{B_{R}^{c}} \Phi'(|u_{0}|) |w|.
		\end{equation}
		 (\ref{C3}) and (\ref{C2})  imply
		\begin{eqnarray}
		\limsup_{ \lambda \downarrow 0} \int_{B_{R}^{c}}\frac{|F(x,u_{0}+\lambda w)- F(x,u_{0})|}{\lambda}&\leq& C_{1} \limsup_{ \lambda \downarrow 0} \int_{B_{R}^{c}} \Phi'(\theta_{\lambda}) |w| \nonumber \\
		&=& C_{1}\lim_{\lambda\rightarrow 0^{+}} \int_{B_{R}^{c}} \Phi'(\theta_{\lambda}) |w| \nonumber \\
		&=&C_{1}\int_{B_{R}^{c}} \Phi'(|u_{0}|) |w|, \;\forall\; R>0, \nonumber
		\end{eqnarray}
		that is,
		\begin{equation*}
		\limsup_{ \lambda \downarrow 0} \int_{B_{R}^{c}}\frac{|F(x,u_{0}+\lambda w)- F(x,u_{0})|}{\lambda}\leq C_{1}\int_{B_{R}^{c}} \Phi'(|u_{0}|) |w|,\;\forall\; R>0.
		\end{equation*}
	\end{proof}

	For the next result we need to fix some notations. In what follows, for each $R>0$, we set
	\begin{eqnarray}
	\Psi_{R}:&L^{\Phi}(B_{R}(0))&\longrightarrow \mathbb{R} \nonumber \\
	&w&\longmapsto \Psi_{R}(w)= \int_{B_{R}(0)} F(x,w). \nonumber
	\end{eqnarray}
	Furthermore, for each $\psi \in L^{\Phi}(B_{R}(0))$, let us consider the function $\tilde{\psi} \in L^{\Phi}(\mathbb{R}^{N})$ given by
	\begin{eqnarray} 
	\tilde{\psi}(x)=\left\{\begin{array}{c}
	\psi(x),\;\; x\in B_{R}(0) \nonumber \\
	0,\;\; x \in B_{R}^{c}(0). \nonumber
	\end{array}
	\right.
	\end{eqnarray} 
	
	\begin{proposition}\label{7}
	Assume $(f_{1})-(f_{3})$. If $(u_{n}) \subset H^{1}(\mathbb{R}^{N})$  is such that 
		$u_{n}\rightharpoonup u_{0}$ in $H^{1}(\mathbb{R}^{N})$ and $\rho_{n} \in \partial \Psi(u_{n})$ satisfies $\rho_{n} \stackrel{\ast}{\rightharpoonup} \rho_{0}$ in $(H^{1}(\mathbb{R}^{N}))^{*}$, then $\rho_{0} \in \partial \Psi(u_{0})$. 
	\end{proposition}
	\begin{proof}
		
		Hereafter,	for each $R>0$, we denote by $u_{n,R}, \rho_{n,R}, u_{0,R}$ and $\rho_{0, R}$ the restriction of the functions $u_{n}, \rho_{n}, u_{0}$ and $\rho_{0}$ to $B_{R}=B_R(0)$ respectively.
		
		For each $\psi \in  L^{\Phi}(B_{R})$, a simple computation yields
		\begin{equation*}
		\int_{B_{R}} \rho_{n,R}\; \psi = \int_{\mathbb{R}^{N}} \rho_{n}\; \tilde{\psi}.
		\end{equation*}
		In addition, $\Psi_{R}^{0}(u_{n,R}, \psi)= \Psi^{0}(u_{n}, \tilde{\psi})$. In fact, 
		\begin{eqnarray}
		\Psi_{R}^{0}(u_{n,R}, \psi)&=& \limsup_{h\rightarrow 0, \lambda \downarrow 0}\frac{1}{\lambda} [\Psi_{R}(u_{n,R}+h+ \lambda \psi )-\Psi_{R}(u_{n,R}+h)] \nonumber \\
		&=& \limsup_{h\rightarrow 0, \lambda \downarrow 0} \frac{1}{\lambda} \left[\int_{B_{R}}[F(x,u_{n,R}+h+ \lambda \psi)-F(x,u_{n,R}+h)] \right] \nonumber \\
		&=& \limsup_{\tilde{h}\rightarrow 0, \lambda \downarrow 0} \frac{1}{\lambda} \left[\int_{\mathbb{R}^{N}}[F(x, u_{n}+\tilde{h}+ \lambda \tilde{\psi})-F(x,u_{n}+\tilde{h})] \right] \nonumber \\
		&\leq & \Psi^{0}(u_{n}, \tilde{\psi}). \nonumber
		\end{eqnarray}
		Analogously $\Psi^{0}(u_{n}, \tilde{\psi}) \leq \Psi_{R}^{0}(u_{n,R}, \psi)$, and the equality $\Psi_{R}^{0}(u_{n,R}, \psi)= \Psi^{0}(u_{n}, \tilde{\psi})$ is proved.
		
		Knowing $\rho_{n} \in \partial \Psi (u_{n})$, we must have 
		\begin{equation*}
		\int_{\mathbb{R}^{N}} \rho_{n} v \leq \Psi^{0}(u_{n}, v),\;\;\forall\; v \in L^{\Phi}(\mathbb{R}^{N}), 
		\end{equation*} 
and so, 
		\begin{equation*}
		\int_{\mathbb{R}^{N}} \rho_{n} \tilde{\psi} \leq \Psi^{0}(u_{n}, \tilde{\psi}), 
		\end{equation*} 
		that is,
		\begin{equation*}
		\int_{B_{R}} \rho_{n, R} \psi \leq \Psi_{R}^{0}(u_{n, R}, \psi), \;\;\forall\;  \psi \in  L^{\Phi}(B_{R}).
		\end{equation*} 
		This shows $ \rho_{n, R} \in \partial \Psi_{R}(u_{n, R})$ for all $n \in \mathbb{N}$.
		
		By hypothesis $\rho_{n} \stackrel{\ast}{\rightharpoonup} \rho_{0}$ in $(H^{1}(\mathbb{R}^{N}))^{*}$, that is, 
		$$
		\int_{\mathbb{R}^{N}} \rho_{n}\; v\rightarrow \int_{\mathbb{R}^{N}} \rho_{0} \; v,\;\;\forall\; v \in H^{1}(\mathbb{R}^{N}),
		$$
	then in particular,
		\begin{equation}\label{40}
		\int_{\mathbb{R}^{N}} \rho_{n}\; v\rightarrow \int_{\mathbb{R}^{N}} \rho_{0} \; v,\;\;\forall\; v \in C_{0}^{\infty}(\mathbb{R}^{N}).
		\end{equation}
	Since 
		$$\overline{C_{0}^{\infty}(\mathbb{R}^{N})}^{||\cdot||_{\Phi}}= L^{\Phi}(\mathbb{R}^{N}),$$ 
		for each $w \in L^{\Phi}(\mathbb{R}^{N})$, there is $(v_{k}) \subset C_{0}^{\infty}(\mathbb{R}^{N})$ such that
		\begin{equation}\label{39}
		v_{k}\rightarrow w\;\;\mbox{in}\;\; L^{\Phi}(\mathbb{R}^{N}).
		\end{equation}
	From this, 
		\begin{eqnarray}
		\left \lvert \int_{\mathbb{R}^{N}} \rho_{n} w- \int_{\mathbb{R}^{N}} \rho_{0} w \right\lvert&=& \left \lvert \int_{\mathbb{R}^{N}} \rho_{n}(w-v_{k})- \int_{\mathbb{R}^{N}} \rho_{0}(w-v_{k}) + \int_{\mathbb{R}^{N}}  v_{k}(\rho_{n} -\rho_{0}) \right\lvert  \nonumber \\
		&\leq& 2 ||w-v_{k}||_{L^{\Phi}(\mathbb{R}^{N})}(||\rho_{n}||_{L^{\tilde{\Phi}}(\mathbb{R}^{N})}+ ||\rho_{0}||_{L^{\tilde{\Phi}}(\mathbb{R}^{N})}) + \left \lvert\int_{\mathbb{R}^{N}}  v_{k}(\rho_{n} -\rho_{0}) \right\lvert,  \nonumber
		\end{eqnarray}
		that is, 
		\begin{equation*}
		\left \lvert \int_{\mathbb{R}^{N}} \rho_{n} w- \int_{\mathbb{R}^{N}} \rho_{0} w \right\lvert  \leq  2 ||w-v_{k}||_{L^{\Phi}(\mathbb{R}^{N})}(M+ ||\rho_{0}||_{L^{\tilde{\Phi}}(\mathbb{R}^{N})}) +  \left \lvert\int_{\mathbb{R}^{N}}  v_{k}(\rho_{n} -\rho_{0}) \right\lvert ,
		\end{equation*}
		where $||\rho_{n}||_{L^{\tilde{\Phi}}(\mathbb{R}^{N})} \leq M$ for all $n \in \mathbb{N}$. 
		
		On the other hand, given $\varepsilon>0$, by (\ref{39}), we can fix $k \in \mathbb{N}$ such that 
		\begin{equation}\label{41}
		||w-v_{k}||_{L^{\Phi}(\mathbb{R}^{N})}(M+ ||\rho_{0}||_{L^{\tilde{\Phi}}(\mathbb{R}^{N})})<\frac{\varepsilon}{2}.
		\end{equation}
		For fixed $k \in \mathbb{N}$ satisfying (\ref{41}), from (\ref{40}), there exists $n_{0}(k) \in \mathbb{N}$ such that
		\begin{equation}\label{42}
		\left \lvert\int_{\mathbb{R}^{N}}  v_{k}(\rho_{n} -\rho_{0}) \right\lvert <\frac{\varepsilon}{2},\;\;\forall\; n\geq n_{0}(k).
		\end{equation}
		Accordingly, from (\ref{41}) and (\ref{42}), 
		\begin{equation*}
		\left \lvert \int_{\mathbb{R}^{N}} \rho_{n} w- \int_{\mathbb{R}^{N}} \rho_{0} w \right\lvert< \frac{\varepsilon}{2} + \frac{\varepsilon}{2}=\varepsilon, \;\;\mbox{for}\;\; n \geq n_{0}(k). \nonumber
		\end{equation*}
		Note that, $ \rho_{n, R} \stackrel{\ast}{\rightharpoonup} \rho_{0,R}$ in $(L^{\Phi}(B_{R}))^{*}$. In fact, given $\psi \in L^{\Phi}(B_{R})$
		\begin{eqnarray}
		\int_{B_{R}} \rho_{n, R}\;\psi \rightarrow \int_{B_{R}} \rho_{0, R} \; \psi &\Leftrightarrow& \int_{\mathbb{R}^{N}} \rho_{n} \tilde{\psi} \rightarrow \int_{\mathbb{R}^{N}} \rho_{0}\; \tilde{\psi}. \nonumber 
		\end{eqnarray}
	As $u_{n, R}\rightarrow u_{0,R}\;\; \mbox{in}\;\; L^{\Phi}(B_{R})$ and $\rho_{n, R}\stackrel{\ast}{\rightharpoonup} \rho_{0,R}\;\;\mbox{in}\;\; L^{\tilde{\Phi}}(B_{R})$ with $ \rho_{n, R} \in \partial \Psi_{R}(u_{n, R})$, by Proposition \ref{38}, 
		$$\rho_{0,R} \in  \partial \Psi_{R}(u_{0,R}).$$
	By definition of generalized gradient, we have 
	\begin{equation}\label{1}
	\left<\rho_{0,R}, v\right> \leq \Psi_{R}^{\circ}(u_{0,R},v),\;\forall\; v \in L^{\Phi}(B_{R}).
	\end{equation}
	\begin{claim} \label{ZZZ1}
		$$\lim_{R\rightarrow +\infty}\Psi_{R}^{\circ}(u_{0,R},w) \leq \Psi^{\circ}(u_{0},w),\;\forall\; w \in L^{\Phi}(\mathbb{R}^{N}).$$
	\end{claim}
	Indeed, given $w\in L^{\Phi}(\mathbb{R}^{N})$ and $R>0$, we obtain
	\begin{eqnarray}
	\Psi_{R}^{0}(u_{0,R}, w)&=& \limsup_{h\rightarrow 0, \lambda \downarrow 0}\frac{1}{\lambda} [\Psi_{R}(u_{0,R}+h+ \lambda w )-\Psi_{R}(u_{0,R}+h)] \nonumber \\
	&=& \limsup_{h\rightarrow 0, \lambda \downarrow 0} \frac{1}{\lambda} \left[\int_{B_{R}}[F(x,u_{0,R}+h+ \lambda w)-F(x,u_{0,R}+h)] \right] \nonumber 
	\end{eqnarray}
	
Setting $\tilde{h}:\mathbb{R}^{N}\rightarrow \mathbb{R}$ as being $\tilde{h}(x)=h(x)\chi_{B_{R}}(x)$, it follows that, 
	\begin{eqnarray}
	&&\int_{B_{R}}F(x, u_{0,R}+h+\lambda w)= \int_{\mathbb{R}^{N}} F(x, u_{0}+\tilde{h}+\lambda w)- \int_{B_{R}^{c}}F(x, u_{0}+\lambda w)\nonumber \\
	&&\mbox{and}\nonumber \\
	&&\int_{B_{R}}F(x, u_{0,R}+h)= \int_{\mathbb{R}^{N}} F(x, u_{0}+\tilde{h})- \int_{B_{R}^{c}}F(x, u_{0}).\nonumber 
	\end{eqnarray}
Therefore, 
	\begin{eqnarray}
	\Psi_{R}^{0}(u_{0,R}, w) &\leq & \limsup_{\tilde{h}\rightarrow 0, \lambda \downarrow 0} \frac{1}{\lambda} \left[\int_{\mathbb{R}^{N}}F(x,u_{0}+\tilde{h}+ \lambda w)-\int_{\mathbb{R}^{N}} F(x,u_{0}+\tilde{h})\right]+ \nonumber\\
	&+& \limsup_{\tilde{h}\rightarrow 0, \lambda \downarrow 0} \left[\int_{B_{R}^{c}}\frac{[F(x,u_{0})-F(x,u_{0}+ \lambda w)]}{\lambda}\right] \nonumber \\
	&=&\limsup_{\tilde{h}\rightarrow 0, \lambda \downarrow 0} \frac{1}{\lambda} \left[\int_{\mathbb{R}^{N}}F(x,u_{0}+\tilde{h}+ \lambda w)-\int_{\mathbb{R}^{N}} F(x,u_{0}+\tilde{h})\right]+ \nonumber\\
	&+& \limsup_{ \lambda \downarrow 0} \left[\int_{B_{R}^{c}}\frac{[F(x,u_{0})-F(x,u_{0}+ \lambda w)]}{\lambda}\right]. \nonumber
	\end{eqnarray}
	By Lemma \ref{C1},
	\begin{equation*}
	\lim_{R\rightarrow +\infty}\left(\limsup_{ \lambda \downarrow 0} \left[\int_{B_{R}^{c}}\frac{[F(x,u_{0})-F(x,u_{0}+ \lambda w)]}{\lambda}\right]\right)=0.
	\end{equation*}
	Thereby,
	\begin{eqnarray}
	\lim_{R\rightarrow +\infty}\Psi_{R}^{\circ}(u_{0,R},w) &\leq &\lim_{R\rightarrow +\infty}\left(\limsup_{\tilde{h}\rightarrow 0, \lambda \downarrow 0} \frac{1}{\lambda} \left[\int_{\mathbb{R}^{N}}F(x,u_{0}+\tilde{h}+ \lambda w)-\int_{\mathbb{R}^{N}} F(x,u_{0}+\tilde{h})\right] \right) \nonumber \\
	&\leq& \lim_{R\rightarrow +\infty} \Psi^{\circ}(u_{0},w)=\Psi^{\circ}(u_{0},w),\nonumber
	\end{eqnarray}
which completes the proof of the claim.
	
	By (\ref{1}) and Claim \ref{ZZZ1},
	\begin{equation*}
	\lim_{R\rightarrow +\infty} \left<\rho_{0,R}, w\right> \leq \Psi^{\circ}(u_{0},w),\;\forall\; w \in L^{\Phi}(\mathbb{R}^{N}).
	\end{equation*}
    Once $\rho_{0, R} \to \rho_{0}$ in $L^{\tilde{\Phi}}(\mathbb{R}^N)$ as $R\rightarrow +\infty$, we conclude
	$$\left<\rho_{0}, w\right> \leq \Psi^{\circ}(u_{0},w),\;\forall\; w \in L^{\Phi}(\mathbb{R}^{N}),$$
	that is, $\rho_{0} \in \partial \Psi(u_{0})\subset L^{\tilde{\Phi}}(\mathbb{R}^{N})$, as asserted.	
	\end{proof}

	\section{Generalized linking theorem}

	The main goal this section is to prove a version of the seminal linking theorem developed in Kryszewski and Szulkin \cite{Kryszewski} ( see also \cite[Chapter 6]{MW} ), by supposing that the functional is only Locally Lipschitz.

	In what follows  $Y$ denotes a Hilbert space that has a total orthonormal sequence denoted by $(e_{k})$.  Using that sequence, let us define the norm
	\begin{eqnarray}
	||\cdot||_0: &Y& \longrightarrow \mathbb{R} \nonumber \\
	&u&\longmapsto ||\cdot ||_0=\sum_{k=1}^{\infty} \frac{1}{2^{k}} |(u, e_{k})|, \nonumber
	\end{eqnarray}
	where $(\cdot, \cdot)$ denotes the inner product in $Y$. From definition $||\cdot||_0$ it follows that  
	\begin{equation*}
	||\cdot||_0 \leq ||u||\;,\; \forall\; u \in Y.
	\end{equation*}

	The topology on $Y$ generated by $||\cdot||_0$ will be denoted by $\sigma$ and all topological notions related to it will include this symbol.
	
	In the sequel we recall some results involving the $\sigma$ topology that can be found in \cite{Kryszewski}. 
	
	\begin{proposition}\label{30}
		If $(u_{n})$ is bounded in $Y$, then
		$$u_{n}\rightharpoonup u\;\;\mbox{in}\;\; Y\Leftrightarrow  u_{n}\stackrel{\sigma}{\rightarrow} u\;\;\mbox{in}\;\;Y.$$
	\end{proposition}

	\begin{definition}\textbf{(Admissible map)}\\
		Let $U$ be an open bounded subset of $Y$ such that  $\overline{U}$ is $\sigma$-closed. A map $g:\overline{U}\rightarrow Y$ is admissible if:
		\begin{itemize}
			\item [(a)] $0 \notin g(\partial U)$;
			\item [(b)] $g$ is $\sigma$-continuous;
			\item [(c)] each point $u \in U$ has a $\sigma$-neighborhood $\mathcal{N}_{u}$ such that $(I_{d}-g)(\mathcal{N}_{u} \cap U)$ is contained in a finite-dimensional subspace of $Y$.
		\end{itemize}
	\end{definition}
	Let $g: \overline{U}\rightarrow Y$ be admissible. Since $\{0\}$ is $\sigma$-closed of $Y$ and $g$ is $\sigma$-continuous, we infer that $g^{-1}(\{0\})$ is $\sigma$-compact.

	For each $u \in g^{-1}(\{0\})$, consider $\sigma$-neighborhood $\mathcal{N}_{u}$ of $u$ such that $(I_{d}-g)(\mathcal{N}_{u}\cap U)$ is contained in a finite-dimensional subspace of $Y$. Note that
	$$\bigcup_{u \in g^{-1}(\{0\})}\mathcal{N}_{u}$$
	is a $\sigma$-open covering of $g^{-1}(\{0\})$. As  	$g^{-1}(\{0\})$ is $\sigma$-compact, there exist $u_{1}, u_{2},..., u_{m} \in g^{-1}(\{0\})$ such that
	$$g^{-1}(\{0\}) \subset \bigcup_{j=1}^{m} (\mathcal{N}_{u_{j}} \cap U)= V.$$
	In addition, $V$ is open and there exists an finite-dimensional subspace $W$ of $Y$ such that
	$$(I_{d}-g)(V) \subset W.$$ 
	The degree, for an admissible map $g$ in $0$ concerning $U$, is defined by
	$$d(g,U,0)=d_{B}(g\mid_{V\cap W}, V \cap W, 0),$$
	where $d_{B}$ is the Brouwer degree.
	
	\begin{proposition}
		The degree of admissible map is well defined.
	\end{proposition}
	\begin{definition}
		A map $h:[0,1] \times \overline{U}\rightarrow Y$ is an admissible homotopy if:
		\begin{itemize}
			\item [(a)] $0 \notin h([0,1] \times \partial U)$;
			\item [(b)] $h$ is continuous in $[0,1] \times (Y,\sigma)$ endowed with the norm $\|(t,u)\|_*=|t|+||u||_0$;
			\item [(c)] each point $(t,u) \in [0,1] \times U$ has neighborhood $\mathcal{N}_{(t,u)}$ with relation the norm $\|(\cdot, \cdot)\|_*$, such that
			$$\{v-h(s,v)\;;\; (s,v) \in \mathcal{N}_{(t,u)} \cap [0,1]\times U\} \subset W$$
			where $W$ is a subspace of $Y$ with $dim W<\infty.$
		\end{itemize}
	\end{definition}
	
	\begin{theorem}
		Let $g:\overline{U}\rightarrow Y$ be an admissible map. Then,
		\begin{itemize}
			\item [(a)]\textbf{(Normalization)}\\ If $v \in U$, then $d(id-v, U,0)=1$;
			\item [(b)] \textbf{(Existence)}\\ if $d(g,U,0) \neq 0$, then $0 \in g(U)$;
			\item [(c)] \textbf{(Homotopy invariance)}\\ If $h$ is admissible homotopy, then $d(h(t, \cdot), U, 0)$ is independent of $t \in [0,1]$.
		\end{itemize}
	\end{theorem}
	
From now on, $X$ is a Hilbert space with $X=Y \oplus Z$, where $Y$ is a separable closed subspace of $X$ and $Z=Y^{\perp}$. If $u \in X$,  $u^{+}$ and $u^{-}$ denote the orthogonal projections from $X$ in $Z$ and in $Y$, respectively. In $X$ let us define the norm
	\begin{eqnarray}
	|||\cdot|||: &X &\longrightarrow \mathbb{R} \nonumber \\
	&u& \longmapsto |||u|||=\max\left\{||u^{+}||, \sum_{k=1}^{\infty}\frac{1}{2^{k}}|(u^{-},e_{k})|\right\}, \nonumber
	\end{eqnarray}
	where $(e_{k})$ is a total orthonormal sequence in $Y$. The topology on $X$ generated by $|||\cdot|||$ will be denoted by $\tau$ and all topological notions related to it will include this symbol.
	
	Observe that for each $u \in X$,
	$$||u^{+}||\leq |||u||| \leq ||u||.$$
	\begin{lemma}\label{36}
	Let $(u_{n}) \subset X$ be a bounded sequence. Then, 
		$$u_{n}\stackrel{\tau}{\rightarrow} u\;\; \mbox{in}\;\; X\Leftrightarrow u_{n}^{-}\rightharpoonup u^{-}\;\; \mbox{and}\;\; u_{n}^{+}\rightarrow u^{+} \quad \mbox{in} \quad X.$$
	\end{lemma}

	\subsection{A special deformation lemma}
	
	The deformation lemma that will prove in this section, see Lemma \ref{26}, completes the study made in \cite{Kryszewski}, in the sense that a similar result was proved in that paper by supposing that the functional $I \in C^{1}(X,\mathbb{R})$.

	Hereafter, $I:X \to \mathbb{R}$ is a {\it locally Lipschitz functional } that is $\tau$-upper semicontinuous.  
		\begin{lemma}(see \cite[Lemma 3.3]{Chang1} ) \label{11}
	Let $u \in X$  and $\varepsilon>0$ such that $\lambda_{I}(u) \geq \varepsilon$. Then, there are $\overline{\varepsilon}>0$ and $\chi_{u} \in X$ with $||\chi_{u}||=1$ such that
	$$\left<l,\chi_{u}\right> \geq \frac{\overline{\varepsilon}}{2}\;,\; \forall\; l \in \partial I(u).$$
	\end{lemma}

In order to prove our next result, we will assume the following condition on $I$: \\

\noindent $(H)$:  If $(u_{n}) \subset I^{-1}([\alpha, \beta])$ is such that $u_{n}\stackrel{\tau}{\rightarrow} u_{0}$ in $X$, then there exists $M>0$ such that $\partial I(u_{n}) \subset B_{M}(0) \subset X^{*},\;\forall\; n \in \mathbb{N}$. In addition, if  $\xi_{n} \in \partial I(u_{n})$ with $\xi_{n} \stackrel{\ast}{\rightharpoonup} \xi_{0}$ in  $X^{*}$, we have $\xi_{0} \in \partial I(u_{0})$.

	\begin{theorem}\label{18}
Assume $(H)$ and let $\alpha < \beta $ and $\varepsilon>0$ such that
	$$
	\lambda_{I}(u)\geq \varepsilon\;,\; \forall\; u \in I^{-1}([\alpha, \beta]).
	$$ 
Then, for each $u_{0} \in I^{-1}([\alpha, \beta])$, there exists $\eta_{0}>0$ such that
		$$\left<\xi,\chi_{u_{0}}\right>> \frac{\overline{\varepsilon}}{3}\;,\; \forall\; \xi \in  \partial I(u)\;,\; u\in B_{\eta_{0}}(u_{0}) \cap I^{-1}([\alpha, \beta]),$$ 
		where $B_{\eta_{0}}(u_{0})=\left\{u \in X\;;\; |||u-u_{0}|||<\eta_{0}\right\}$ with  $\chi_{u_{0}}$ and $\overline{\varepsilon}>0$ given in Lemma \ref{11}.
	\end{theorem}
	\begin{proof}
		
	Arguing by contradiction, assume that there exists $(u_{n}) \subset I^{-1}([\alpha, \beta])$ with $u_{n}  \stackrel{\tau}{\rightarrow} u_{0}$ in $X$ and 
		\begin{equation}\label{9}
		\left<\xi_{n},\chi_{0}\right> \leq  \frac{\overline{\varepsilon}}{3}\;, \forall\; n \in \mathbb{N},
		\end{equation}
		where $\xi_{n} \in \partial I(u_{n})$. By condition $(H)$, going to a subsequence if necessary, there is $\xi_{0}  \in \partial I(u_{0})$ such that 
		\begin{equation}\label{10}
		\xi_{n} \stackrel{\ast}{\rightharpoonup} \xi_{0}\;\; \mbox{in}\;\; X^{*}.
		\end{equation}
		Therefore, from (\ref{9}) and (\ref{10}), 
		\begin{equation*}
		\left<\xi_{0},\chi_{0}\right> \leq  \frac{\overline{\varepsilon}}{3},
		\end{equation*}
   contrary to Lemma \ref{11}.
	\end{proof}
	
	\begin{lemma}\label{20}
		Under the assumptions of Theorem \ref{18}, there exists a $\tau$-open neighborhood $V$ of $I^{\beta}=\{u \in X\,:\, I(u) \leq \beta\}$ and a vector field $P:V\rightarrow X$ satisfying:
	\end{lemma}
	\begin{itemize}
		\item [(a)] $P$ is locally Lipschitz continuous and  $\tau$-locally Lipschitz continuous,
		\item [(b)] each point $u \in V$ has a $\tau$-neighborhood $V_{u}$ such that $P(V_{u})$ is contained in a finite-dimensional subspace of $X$,
		\item [(c)] $m=\displaystyle \sup_{u \in V} ||P(u)|| \leq 1$ and $\left<\xi,P(u)\right>\geq 0,\;\forall\; \xi\in \partial I(u)$ and $u \in V$,
		\item [(d)] for each $u \in I^{-1}([\alpha, \beta])$,  
		$$\left<\xi,P(u)\right>> \frac{\overline{\varepsilon}}{3},\;\forall\; \xi \in \partial I(u)$$
		for some $\overline{\varepsilon}>0$.
	\end{itemize}
	
	\begin{proof}
		
		For each $u_{0} \in I^{-1}([\alpha, \beta])$, by  Theorem \ref{18}, there exists $\eta_{0}>0$ such that
		\begin{equation}\label{16}
		\left<\xi,\chi_{u_{0}}\right>\geq \frac{\overline{\varepsilon}}{3}\;,\; \forall\; \xi \in  \partial I(u)\;,\; u\in B_{\eta_{0}}(u_{0}).
		\end{equation}
		Since $I$ is $\tau$-upper semicontimuous, 
		$$
		\tilde{\mathcal{N}}=I^{-1}((-\infty, \alpha))
		$$ is $\tau$-open in $X$, and 
		$$
		\mathcal{N}=\left\{B_{\eta_{u}}(u)\;,\; u \in I^{-1}([\alpha, \beta]) \right\} \cup \left\{\tilde{\mathcal{N}}\right\}
		$$
		is a $\tau$-neighborhood for $I^{\beta}$. In addition $(I^{\beta}, \tau)$ is a metric space, then there exists a  $\tau$-locally finite $\tau$-open covering  $\mathcal{V}=\left\{\mathcal{V}_{i}\;:\; i \in \mathcal{J}\right\}$ of $I^{\beta}$ (see \cite{Ryszard}) more fine than $\mathcal{N}$. We define the $\tau$-open neighborhood of $I^{\beta}$ by 
		$$
		V= \bigcup_{i \in \mathcal{J}} \mathcal{V}_{i}
		$$
		and set $\left\{\gamma_{i}\;:\; i \in \mathcal{J}\right\}$ as being  a $\tau$-Lipschitz continuous partition of unity subordinated to $\mathcal{M}$. Employing the notations above, we set the vector field $P:V \to X$ by 
	$$
	 P(u)=\sum_{i \in \mathcal{J}} \gamma_{i}(u)w_{i}, 
	$$
		where:
		\begin{itemize}
			\item If $\mathcal{V}_{i} \subseteq B_{\eta_{u_i}}(u_i)$, we choose $w_{i}=\chi_{u_i}$ ($\chi_{u_i}$ is given in Lemma \ref{11}).
			\item If $\mathcal{V}_{i}\subseteq \tilde{\mathcal{N}}$, we choose $w_{i}=0$. 
		\end{itemize}
		\begin{itemize}
			\item [(a)] A straightforward  computation  shows $P$ is $\tau$-Locally Lipschitz and Locally Lipschitz.
			\end{itemize}
			\begin{itemize}
			\item [(b)] For each $u \in V$ there exists a $\tau$-open neighborhood $\mathcal{N}_{u}$ of $u$ such that
			$$\mathcal{N}_{u} \cap \mathcal{V}_{i} \neq \emptyset\;\;\mbox{for}\;\; i \in \{1,2,...,k\}$$
		and
			$$P(\mathcal{N}_{u})= \left\{\sum_{i=1}^{k}\gamma_{i}(v) w_{i}\;;\; v \in \mathcal{N}_{u}\right\} \subset span\left\{w_{1}, w_{2},...,w_{k}\right\}:=W$$
			where $dim W<\infty$.
			\end{itemize}
			\begin{itemize}
			\item [(c)] Given $u \in V$
			\begin{eqnarray}
			||P(u)|| &\leq & \sum_{i \in \mathcal{J}} \gamma_{i}(u) ||w_{i}|| \nonumber \\
			&\leq& \sum_{i \in \mathcal{J}} \gamma_{i}(u)=1, \nonumber
			\end{eqnarray}
			implying that $m=\sup_{u \in V} ||P(u)|| \leq 1$. Let $u \in V$ and $\xi \in \partial I(u)$, then by (\ref{16})
			$$(\xi, P(u))= \sum_{i \in \mathcal{J}} \gamma_{i}(u)(\xi, w_{i}) \geq 0.$$
		\end{itemize}
			\begin{itemize}
			\item  [(d)] For each $u \in I^{-1}([\alpha, \beta])$,  $u \in suppt(\gamma_{i}), i \in \mathcal{J}_u$ ($\mathcal{J}_u$ is finite). Since $(suppt(\gamma_{i}))_{i \in \mathcal{J}}$ is subordinated to $\mathcal{M}$, there exists $\mathcal{V}_{i} \in \mathcal{V}$ such that
			\begin{equation}\label{17}
			u \in suppt(\gamma_{i}) \subseteq \mathcal{V}_{i}\;,\;\forall\; i \in \mathcal{J}_u.
			\end{equation}
		For each $i \in \mathcal{J}_u$, there exists $v_{i} \in I^{-1}([\alpha, \beta])$ such that
		$$
		\mathcal{V}_{i} \subseteq B_{\eta_{v_{i}}}(v_{i})\;,\; \eta_{v_{i}}>0.
		$$
		From this, $w_{i}=\chi_{v_i}$ and by (\ref{16}),
	$$
			(\xi, P(u))= \sum_{i \in \mathcal{J}_u} \gamma_{i}(u)(\xi, \chi_{u_i}) >\frac{\overline{\varepsilon}}{3}\sum_{i \in \mathcal{J}_u} \gamma_{i}(u) =\frac{\overline{\varepsilon}}{3}. \nonumber
	 $$	
			\end{itemize}
	\end{proof}
	
In the sequel, we will consider the following Cauchy problem
	\begin{eqnarray} \label{001}
	\left\{\begin{array}{c}
	\frac{d}{dt}\eta(t,u) =- P(\eta(t,u)) \\
	\eta(0,u)=u \in I^{\beta}. 
	\end{array}
	\right.
	\end{eqnarray}
	The classical theory of ordinary differential equations asserts (\ref{001}) has a unique solution $\eta(u, \cdot)$ that exists for all $t\geq 0$ with $\eta(t,u) \in V$, because $P$ is a bounded vector field, see Lemma \ref{20}.
	\begin{lemma}\label{53}
		For each $u \in I^{\beta}$ consider a map
		\begin{eqnarray}
		\Gamma_{u}: &\mathbb{R}& \longrightarrow \mathbb{R} \nonumber \\
		&t&\longmapsto 	\Gamma_{u}(t)=I(\eta(t,u)). \nonumber
		\end{eqnarray}
		Then, $	\Gamma_{u} \in Lip_{loc}(\mathbb{R}, \mathbb{R})$ and $t\mapsto 	\Gamma_{u}(t)$ is not increasing.
	\end{lemma}
	\begin{proof}
		
		Since $I \in Lip_{loc}(X, \mathbb{R})$, for each $u \in I^{\beta}$ and $t \in \mathbb{R}$, there exists $\varepsilon>0$ such that
		\begin{eqnarray}
		|I (\eta(s_{1},u))-I(\eta(s_{2},u))| \leq K(\eta(t,u))\; ||\eta( s_{1},u)- \eta(s_{2},u)||,\;\forall\;\eta( s_{1},u), \eta(s_{2},u) \in B_{\varepsilon}(\eta(t,u)). \nonumber
		\end{eqnarray}
		Considering $\Theta =\max\{s_{1}, s_{2}\}$ and $\kappa= \min \{s_{1}, s_{2}\}$, we obtain
		\begin{eqnarray}
		|I(\eta(s_{1},u))-I(\eta(s_{2},u))| &\leq& K(\eta(t,u)) \left\| \int_{\kappa}^{\Theta} -P (\eta(s,u)) ds \right\| \nonumber \\
		&\leq &  K(\eta(t,u))  \int_{\kappa}^{\Theta} ||P (\eta(s,u))|| ds. \nonumber
		\end{eqnarray}
As $\|P(u)\| \leq 1$ for all $u \in V$, 
		$$
		|I(\eta(s_{1},u))-I(\eta(s_{2},u))| \leq  K(\eta(t,u)) |s_{1}-s_{2}|,\;\forall\; s_{1}, s_{2} \in  (t-\delta,t+\delta)
		$$
	for $\delta>0$ small enough. This shows $\Gamma_{u} \in Lip_{loc}(\mathbb{R}, \mathbb{R})$.
		
		We will proof that $\Gamma_{u}$ is not increasing. By Lemma \ref{19}, $\Gamma_{u}$ is differentiable a.e in $\mathbb{R}$ and
		$$
		\Gamma_{u}'(t) \leq \max\left\{\left<\xi, \frac{d}{dt} \eta(t,u)\right>\;;\; \xi \in \partial I(\eta(t,u))\right\} \quad \mbox{a.e in} \quad \mathbb{R},
		$$
or equivalently
		$$
		\Gamma_{u}'(t) \leq - \min\left\{\left<\xi, P(\eta(t,u)) \right>\;;\; \xi \in \partial I(\eta(t,u))\right\} \quad \mbox{a.e in} \quad \mathbb{R}.
		$$
Thereby, by Lemma \ref{20},  
		\begin{equation}\label{21}
		\Gamma_{u}'(t) \leq - \frac{\overline{\varepsilon}}{3}\leq 0, \quad a.e \quad \mbox{in} \quad  \mathbb{R}, 
		\end{equation}
	from where it follows that desired result. 
	\end{proof}

	\begin{lemma}(Deformation lemma)\label{26}
		Under the assumptions of Theorem \ref{18}, the vector field $\eta: \mathbb{R}^{+}\times I^{\beta}\rightarrow X$ given in (\ref{001}) is well defined and satisfies the following properties:
	\end{lemma}
	\begin{itemize}
		\item [(a)] There exists $T>0$ such that 
		$$\eta(T, I^{\beta}) \subset I^{\alpha};$$
		\item [(b)] Each point $(t,u) \in [0,T] \times I^{\beta}$ has a $\tau$-neighborhood $N_{(t,u)}$ such that
		$$\left\{v-\eta (s,v) \in N_{(t,u)}\cap  ([0,T] \times I ^{\beta}) \right\} $$
		is contained in a finite-dimensional subspace of $X$;
		\item [(c)] $\eta$ is continuous in  $[0,+\infty) \times (X,\tau)$ endowed with the norm $\|(t,u)\|_{\bigstar}=|t|+|||u|||$.
	\end{itemize}
	\begin{proof}

		\begin{itemize}
			\item [(a)] Given $u \in I^{\beta}$ and choosing $T=3\;\left(\frac{\beta-\alpha}{\overline{\varepsilon}}\right)>0$, let us look at the following cases:
			\item (i) \, There is $t_{0} \in [0,T]$ such that $I(\eta(t_{0},u))<\alpha$. 
			
			Since $t\mapsto \Gamma_{u}(t)$ not increasing, then
			$$I(\eta(T,u)) \leq I(\eta(t_{0},u))<\alpha,$$
		proving that $\eta(T,u) \in I^{\alpha}$.
			\item (ii) \, $\eta(t, u) \in I^{-1}([\alpha, \beta]), \forall\; t \in [0,T]$. 
			\begin{claim}\label{07}
				For each $u \in I^{\beta}$, $I(u)- I(\eta(T,u)) \geq \frac{\overline{\varepsilon }\,T}{3}.$
			\end{claim}
			In fact, recalling that
			$$
			\Gamma_{u}'(t) \leq -\frac{\overline{\varepsilon}}{3}\;,\; \eta(t,u) \in I^{-1}([\alpha, \beta]),\;\forall\; t \in [0,T], \quad \mbox{(see (\ref{21}))}
			$$
		we find 
		$$
			-T \;\frac{\overline{\varepsilon}}{3}= - \int_{0}^{T} \frac{\overline{\varepsilon}}{3} dt \geq \int_{0}^{T} 	\Gamma_{u}'(t) dt = 	\Gamma_{u}(T)-	\Gamma_{u}(0)=I(\eta(T,u))- I(u) 
		$$
			%(ver Teorema \ref{22}).
as required. From this,  
	$$
			I(\eta(T,u)) \leq I(u)- 
			T\;\frac{\overline{\varepsilon}}{3} \leq  \beta - \frac{3}{\overline{\varepsilon}}(\beta-\alpha)\frac{\overline{\varepsilon}}{3}= \alpha.  
	$$
	
For $(b)$ and $(c)$, see \cite[Proposition 2.2]{Kryszewski} (or \cite[Lemma 6.8]{MW} ).
		\end{itemize}  
	\end{proof}

\subsection{Generalized linking theorem}
	
	Let $Y$ be a separable closed subspace of a Hilbert space $X$ and $Z=Y^{\perp}$. If $u \in X$, as in the previous section, $u^{+}$ and $u^{-}$ denote the orthogonal projections in $Z$ and $Y$, respectively. 

Given $\rho >r>0$ and $z \in Z$ with $||z||=1$, we set 
\begin{eqnarray}
&&\mathcal{M}=\left\{u=y+\lambda z\;; ||u||\leq \rho, \lambda \geq 0 \;\;\mbox{and}\;\; y \in Y\right\} \nonumber \\
&& \mathcal{M}_{0}=\left\{u=y+\lambda z\;;y \in Y, ||u||= \rho\;\;\mbox{and}\;\; \lambda \geq 0 \;\;\mbox{or}\;\;  ||u||\leq  \rho\;\;\mbox{and}\;\; \lambda=0 \right\} \nonumber \\
&& S=\left\{u \in Z\;; ||u||=r\right\}. \nonumber
\end{eqnarray} 
Assume $I \in Lip_{loc}(X, \mathbb{R})$ such that
\begin{equation}\label{5}
I \;\mbox{is}\;\;\tau-\mbox{upper semicontinuous}
\end{equation}
and
\begin{equation}\label{6}
b=\inf_{S}I>0= \sup_{\mathcal{M}_{0}}I\;,\; d=\sup_{\mathcal{M}} I < \infty
\end{equation}

\begin{theorem}\label{59}\textbf{}\\
	Assume $I\in Lip_{loc}(X, \mathbb{R})$, (\ref{5}), (\ref{6}) and $(H)$. Then, there is $c \in [b,d]$ and a sequence $(u_{n}) \subset X$ such that
	$$I(u_{n}) \rightarrow c\;\;\mbox{and}\;\; \lambda_{I}(u_{n})\rightarrow 0.$$
\end{theorem}
\begin{proof}
	
	If the conclusion of the theorem is not true, then there is $\varepsilon>0$ such that
	\begin{equation}\label{25}
	\lambda_{I}(u) \geq \varepsilon\;,\; \forall\; u \in I^{-1}([b-\varepsilon, d+\varepsilon]).
	\end{equation}
	By (\ref{5}), (\ref{6}) and (\ref{25}), the hypotheses of Lemma \ref{26} are satisfied with $\alpha=b-\varepsilon$ and $\beta=d+\varepsilon$. 
	%Analogously the proof in \cite{MW} of the Theorem 6.10, we get a contradiction. 
	Since $\mathcal{M} \subset I^{d+\varepsilon}$, because $d=\sup_{\mathcal{M}} I$, it follows from Lemma \ref{26} that there exists $T>0$ such that
	\begin{equation}\label{32}
	\eta(T, \mathcal{M}) \subset I^{b-\varepsilon}. 
	\end{equation}
	Let $U$ be the interior of $\mathcal{M}$ in $E=Y\oplus \mathbb{R}z$ and define the homotopy
	\begin{eqnarray}
	h: &[0,T] \times \mathcal{M}& \longrightarrow E \nonumber \\
	&(t,u)& \longmapsto h(t,u)= \eta(t,u)^{-}+(||\eta(t,u)^{+}||-r)z, \nonumber
	\end{eqnarray}
	where $\rho>r>0$.	
	\begin{claim}
		$h$ is an admissible homotopy, that is,
		\begin{itemize}
			\item [(a)] $0 \notin h([0,T]\times \partial U)$;
			\item [(b)] $h$ is continuous in $[0,T] \times E$ endowed with the norm $\|(t,u,s)\|_{\blacklozenge}=|t|+||u||_0+|s|$.
			\item [(c)] for each point $(t,u) \in  [0,T] \times U$ there exists an open neighborhood $\mathcal{N}_{(t,u)}$ with relation to the norm $\|(\cdot,\cdot,\cdot)\|_{\blacklozenge}$ such that
			$$\{v-h(s,v)\;;\; (s,v) \in \mathcal{N}_{(t,u)} \cap [0,1]\times U\} \subset \tilde{W}$$
	for some finite-dimensional subspace $\tilde{W}$ of $X$.
		\end{itemize}
	\end{claim}
	\textbf{(a)} Note that
	\begin{eqnarray}
	h(t,u)=0 &\Leftrightarrow& \eta(t,u)^{-}+(||\eta(t,u)^{+}||-r)z=0 \nonumber \\
	&\Leftrightarrow& \eta(t,u)^{-}=(r-||\eta(t,u)^{+}||)z, \nonumber
	\end{eqnarray}
	that is, $r=||\eta(t,u)^{+}||$ and $\eta(t,u)^{-}=0$, from where it follows that $\eta(t,u) \in S$. Suppose 
	$$
	0 \in h([0,T]\times \partial U).
	$$
	Then, there exists $(t,u) \in [0,T] \times \partial U$ such that $h(t,u)=0$, that is, $\eta(t,u) \in S$. Owing the to (\ref{6}), 
	$$
	b>0> \sup_{\partial U} I \geq I(u) \geq I(\eta(t,u)) \geq b,
	$$
	which is absurd.
	
	\textbf{(b)} By definition of $h$ and Lemma \ref{26}, $h$ is continuous with relation to the norm $\|(\cdot, \cdot,\cdot)\|_{\blacklozenge}$.
	
	\textbf{(c)} By Lemma \ref{26}, for each $(t,u) \in  [0,T] \times U$, there exists an open neighborhood $\mathcal{N}_{(t,u)}$ with relation the norm $\|(\cdot, \cdot,\cdot)\|_{\blacklozenge}$ and a subspace $W$ of $E$ with $dim\; W< \infty$,  such that 
	$$\{v-\eta(s,v)\;;\; (s,v) \in \mathcal{N}_{(t,u)} \cap [0,1]\times U\} \subset W.$$
	Thereby, for each $(s,v) \in \mathcal{N}_{(t,u)} \cap ([0,1]\times U)$,
	\begin{eqnarray}
	v-h(s,v)&=& v- [\eta(t,u)^{-}+(||\eta(t,u)^{+}||-r)z] \nonumber \\
	&=& [v-\eta(s,v)]+[\eta(s,v)^{+}+(||\eta(s,v)^{+}||-r)z] \in W\oplus \mathbb{R}z:=\tilde{W}, \nonumber
	\end{eqnarray}
showing that $h$ is an admissible homotopy. By homotopy invariance
	$$d(h(T, \cdot), U, 0)= d(h(0, \cdot), U, 0),$$
	that is, 
	$$d(h(T, \cdot), U, 0)= d(h(Id-rz, U, 0))=1.$$
	Thus, there exists $u \in U$ such that $h(T,u)=0$. From this, $\eta(T,u) \in S$, and so, 
	$$b \leq I(\eta(T,u)).$$
	On the other hand, according to (\ref{32}), $I(\eta(T,u)) \leq b- \varepsilon$, which is absurd.
\end{proof}

\section{Proof of Theorem \ref{Teorema1}}	
	
In order to prove Theorem \ref{Teorema1}, our first step is to show that the energy function $\varphi$ associated with problem $(P)$, see (\ref{varphi1}) or (\ref{varphi2}), satisfies condition $(H)$.  	
\subsection{Condition check (H)}
	Let $(u_{n}) \subset \varphi^{-1}([\alpha, \beta])$ with $u_{n}\stackrel{\tau}{\rightarrow} u_{0}$ in $H^{1}(\mathbb{R}^{N})$. By definition of the norm that generates the $\tau$-topology,  we infer that $u_{n}^{+}\rightarrow u^{+}$ in $H^{1}(\mathbb{R}^{N})$, hence there is $R>0$ such that 
	\begin{equation}\label{34}
	||u_{n}^{+}|| \leq R\;,\; \forall\; n \in \mathbb{N}.
	\end{equation}
	\begin{claim}
		$(u_{n})$ is bounded in $H^{1}(\mathbb{R}^{N})$.
	\end{claim}
	For each $u \in H^{1}(\mathbb{R}^{N})$, 
	$$\varphi(u)= \frac{1}{2}||u^{+}||^{2}-\frac{1}{2}||u^{-}||^{2}- \int_{\mathbb{R}^{N}} F(x,u).$$
	By $(f_{3})$, (\ref{34}) and using the fact that $u_{n} \in \varphi^{-1}([\alpha, \beta])$ for all $ n \in \mathbb{N}$, we discover
	\begin{eqnarray}
	||u_{n}^{-}||^{2}&=& ||u_{n}^{+}||^{2}-2\varphi(u_{n})- 2\int_{\mathbb{R}^{N}} F(x,u_{n}) \nonumber \\
	&\leq & ||u_{n}^{+}||^{2}-2\varphi(u_{n}) \nonumber \\
	&\leq & R^{2}- 2 \alpha. \nonumber
	\end{eqnarray}
	Therefore,
	$$||u_{n}||^{2}= ||u_{n}^{+}||^{2}+||u_{n}^{-}||^{2}\leq R^{2}+ R^{2}- 2 \alpha=2(R^{2}-\alpha),$$
	showing the boundedness of $(u_n)$.
	\begin{claim}
		$\partial \varphi (u_{n})$ is uniformly bounded for all $n \in  \mathbb{N}$.
	\end{claim}
	Given $\xi_{n} \in \partial \varphi (u_{n})$, we have 
	$$
	\xi_{n}= Q'(u_{n})- \rho_{n},\;\;\mbox{with}\;\;\rho_{n} \in \partial \Psi(u_{n}) \subset (L^{\Phi}(\mathbb{R}^{N}))^{*},
	$$
	where
	$$
	Q(u)=\int_{\mathbb{R}^N}(|\nabla u|^{2}+V(x)|u|^{2}), \quad u \in H^{1}(\mathbb{R}^N).
	$$
	For $v \in H^{1}(\mathbb{R}^{N})$, the continuous Sobolev embedding together with  H\"older inequality and $(F_{*})$ leads to 
	\begin{eqnarray}
	|\left<\xi_n, v\right>|&=&|\left<Q'(u_{n}), v\right>- \left<\rho_{n}, v\right>| \nonumber \\
	&=& |(u_{n}^{+},v)-(u_{n}^{-},v) - \left<\rho_{n}, v\right>| \nonumber \\
	&\leq & ||v||\;(||u_{n}^{+}||+||(u_{n}^{-})||)+|\left<\rho_{n}, v\right>|  \nonumber \\
	&\leq & K_{1}\;||v|| +\left\lvert \int_{\mathbb{R}^{N}} \rho_{n}v\right \lvert \nonumber \\
	&\leq & K_{1}\;||v|| +C_{0} \int_{\mathbb{R}^{N}}(|u_{n}|+|u_{n}|^{p-1})|v| \nonumber \\
	&\leq&K_{1}\;||v|| + C_{0}||u_{n}||_{2}||v||_{2}+C_{0}||u_{n}||_{p}^{p-1}||v||_{p} \nonumber\\
	&\leq& K ||v||, \quad \forall n \in \mathbb{N}, \nonumber
	\end{eqnarray}
from where it follows that $||\xi_{n}||_{*} \leq K$ for all $n \in \mathbb{N}$, as asserted.
	\begin{claim}
		If $(\xi_{n}) \subset \partial \varphi(u_{n})$ is such that $\xi_{n} \stackrel{*}{\rightharpoonup} \xi_{0}$ in $(H^{1}(\mathbb{R}^{N}))^{*}$, then $\xi_{0} \in \partial \varphi(u_{0})$.
	\end{claim}
Indeed, since $\xi_{n} \in \partial \varphi(u_{n})$, there is $\rho_{n} \in  \partial\Psi(u_{n})$ such that
	$$\xi_{n}= Q'(u_{n})- \rho_{n}.$$
	As $(u_{n})$ is bounded in $H^{1}(\mathbb{R}^{N})$ and $u_{n}\stackrel{\tau}{\rightarrow} u_{0}$ in $H^{1}(\mathbb{R}^{N})$, we have 
	$$u_{n}\rightharpoonup u_{0}\;\;\mbox{in}\;\; H^{1}(\mathbb{R}^{N}).$$
	Thereby,   
	$$
	\left<\rho_{n}, v\right>= \left<Q'(u_{n}), v\right>- \left<\xi_{n}, v\right>\rightarrow \left<Q'(u_{0}), v\right>- \left<\xi_{0}, v\right>, \quad \forall v \in H^{1}(\mathbb{R}^{N})
	$$
	that is, $\rho_{n} \stackrel{*}{\rightharpoonup} Q'(u_{0})-\xi_{0}$ in  $ (H^{1}(\mathbb{R}^{N}))^{*}.$ From Proposition \ref{7},  
	$$
	Q'(u_{0})-\xi_{0} \in \partial \Psi(u_{0}),
	$$
	finishing the proof of condition $(H)$.  

\vspace{0.5 cm}

Now, we are going to prove the assumptions of the Theorem \ref{59}, which will be done in some lemmas. 	
\begin{lemma}(see \cite[Lemma 6.13]{MW})\label{55}
		Suppose $(V)$ and $(f_{1})-(f_{3})$. Then, there exists $r>0$ such that
		\begin{equation*}
		b:= \inf_{u \in Z, ||u||=r} \varphi(u)>0.
		\end{equation*} 
	\end{lemma}
%	\begin{proof}
%		
%		By $(f_{2})$ there exists $\tilde{C}>0$ such that 
%		$$
%		|F(x,t)| \leq \tilde{C}(|t|^{q}+|t|^{p}) \quad \forall t \in \mathbb{R}.
%		$$
%		Thus, for $u \in Z$, the Sobolev embedding together with $(f_{2})$ gives 
%		\begin{eqnarray}
%		\varphi(u)&=&\frac{1}{2}||u^{+}||^{2}-  F(x,u) \nonumber \\
%		&\geq&  \frac{1}{2}||u^{+}||^{2}-  \tilde{C}(||u||_{q}^{q}+||u||_{p}^{p}) \nonumber \\
%		&\geq&  \frac{1}{2}||u^{+}||^{2}-  M(||u^{+}||^{q}+||u^{+}||^{p}). \nonumber
%		\end{eqnarray}
%		Choosing $0<r=\min\left\{\left(\frac{1}{4M}\right)^{\frac{1}{q-2}}, \left(\frac{1}{8M}\right)^{\frac{1}{p-2}} \right\}$, 
%		$$
%		\varphi(u) \geq \frac{r^{2}}{8}>0, \quad \mbox{for} \quad u \in Z \quad \mbox{with} \quad ||u||=r.
%		$$
%	\end{proof}

	\begin{lemma}\label{60}
		Assume $(f_{1})-(f_{4})$. Then, there are $z \in Z$ and $\rho>r>0$ with $||z||=r$ such that
		\begin{equation*}
		\max_{\mathcal{M}_{0}} \varphi =0 \quad \mbox{and} \quad d=\sup_{\mathcal{M}}\varphi<\infty,
		\end{equation*}
		where 
	$$
	\mathcal{M}=\left\{u=y+\lambda z\;; ||u||\leq \rho, \lambda \geq 0 \;\;\mbox{and}\;\; y \in Y\right\} 
	$$
	and
	$$
	\mathcal{M}_{0}=\left\{u=y+\lambda z\;;y \in Y, ||u||= \rho\;\;\mbox{and}\;\; \lambda \geq 0 \;\;\mbox{or}\;\;  ||u||\leq  \rho\;\;\mbox{and}\;\; \lambda=0 \right\}.
	$$ 
	\end{lemma}
	\begin{proof}
		
	First of all, we recall that $(f_4)$ ensures that for each $\delta>0$, there is $c=c(\delta)>0$ such that 
\begin{equation} \label{54}
	F(x,t) \geq c|t|^{\theta}- \delta |t|^{2},\;\forall\; (x,t) \in \mathbb{R}^{N} \times \mathbb{R}.
\end{equation}
For fixed $z \in Z\backslash \{0\}$, consider $u \in Y\oplus \mathbb{R}^{+}z$, that is, $u=y+\lambda z$ for $\lambda\geq 0$ and $y \in Y$. By (\ref{54}) and continuous Sobolev embedding, 
		\begin{eqnarray}
		\varphi(u)=\varphi(y+\lambda z)&=& \frac{\lambda^{2}}{2}||z||^{2}- \frac{1}{2}||y||^{2}- \int_{\mathbb{R}^{N}} F(x,y+\lambda z) \nonumber \\
		& \leq &  \frac{\lambda^{2}}{2}||z||^{2}- \frac{1}{2}||y||^{2}- c_{1}||y+\lambda z||_{\theta}^{\theta}+ \delta ||y+\lambda z||_{2}^{2} \nonumber \\
		& \leq &  \frac{\lambda^{2}}{2}||z||^{2}- \frac{1}{2}||y||^{2}- c_{1}||y+\lambda z||_{\theta}^{\theta}+ \delta c_{2}||y+\lambda z||^{2} \nonumber \\
		&=& \lambda^{2} \left(\frac{1}{2}+\delta c_{2} \right)||z||^{2}+ \left(\delta c_{2}- \frac{1}{2}\right)||y||^{2}- c_{1} ||y+\lambda z||_{\theta}^{\theta}. \nonumber
		\end{eqnarray}
		Choosing $\delta = \frac{1}{4c_{2}}$, we get 
		\begin{equation}\label{57}
		\varphi(y+\lambda z) \leq \frac{3}{4} \lambda^{2} ||z||^{2}- \frac{1}{4}||y||^{2}-c_{1}||y+\lambda z||_{\theta}^{\theta}.
		\end{equation}
		\begin{claim}\label{01}
			There exists $R>0$ that 
			$$\varphi(u) \leq 0,\;\mbox{for}\;\; u \in Y\oplus \mathbb{R}^{+}z\;\;\mbox{and}\;\; ||u||=R.$$
		\end{claim}
		Suppose that there are $(y_{n}) \subset Y$ and $(\lambda_{n}) \subset [0,+\infty)$ such that $||y_{n}+\lambda_{n} z||\rightarrow +\infty$ and $\varphi(y_{n}+\lambda_{n}z)> 0$ for all $n \in \mathbb{N}$. Thus,
		$$0  < \varphi(y_{n}+\lambda_{n}z)= \frac{\lambda_{n}^{2}}{2}||z||^{2}-\frac{1}{2}||y_{n}||^{2}- \int_{\mathbb{R}^{N}} F(x,y_{n}+\lambda_{n}z)\leq \frac{\lambda_{n}^{2}}{2}||z||^{2}-\frac{1}{2}||y_{n}||^{2},$$
		implying that
		\begin{equation}\label{58}
		\frac{||y_{n}||^{2}}{||\lambda_{n} z||^{2}} \leq 1,\;\forall\; n \in \mathbb{N}.
		\end{equation}
		Since there exists $M>0$ such that
		$$||\lambda_{n} z||_{\theta} \leq M ||y_{n}+\lambda_{n} z||_{\theta}, \quad \forall\; n \in \mathbb{N},$$
we derive that
		\begin{eqnarray}
		\varphi(y_{n}+\lambda_{n} z) &\leq& \frac{3}{4} \lambda_{n}^{2} ||z||^{2}- \frac{1}{4}||y_{n}||^{2}-c_{1}||y_{n}+\lambda_{n} z||_{\theta}^{\theta} \nonumber \\
		&\leq& \frac{3}{4} \lambda_{n}^{2} ||z||^{2}- \frac{1}{4}||y_{n}||^{2}-\frac{c_{1}}{M^{\theta}}||\lambda_{n} z||_{\theta}^{\theta} \nonumber \\
		&=& \frac{3}{4} \lambda_{n}^{2} ||z||^{2}- \frac{1}{4}||y_{n}||^{2}-\lambda_{n}^{\theta}\frac{c_{1}}{M^{\theta}}|| z||_{\theta}^{\theta}, \nonumber
		\end{eqnarray}
		that is,
		\begin{equation}\label{73}
		\varphi(y_{n}+\lambda_{n} z) \leq  \tilde{k}_{1}\lambda_{n}^{2} - \lambda_{n}^{\theta}\tilde{k}_{2}- \frac{1}{4}||y_{n}||^{2},\;\forall\; n \in \mathbb{N}.
		\end{equation}
	As $||y_{n}+ \lambda_{n} z||\rightarrow +\infty$, the following may occur:
		\begin{itemize}
			\item [(i)] $||y_{n}|| \rightarrow +\infty$ and $ \lambda_{n} \rightarrow +\infty$.
			\item [(ii)] $(y_{n})$ is bounded and $\lambda_{n} \rightarrow +\infty$.
			\item [(iii)] $||y_{n}|| \rightarrow +\infty$ and $(\lambda_{n})$ is bounded.
		\end{itemize}
		In any of the above cases, by (\ref{73}), there is $n_{0} \in \mathbb{N}$ such that 
		$$0 \leq \varphi(y_{n}+\lambda_{n} z)<0,\;\forall\; n \geq n_{0},$$
	which is absurd. This proves Claim \ref{01}. 
		
		Fixed $\rho=R>0$, we define $\mathcal{M}_{0}$ and $\mathcal{M}$ as above. For $u \in \mathcal{M}_{0}$, we have the following:
		\begin{itemize}
			\item if $||u||=\rho$ and $\lambda\geq 0$, then $\varphi(u) \leq 0$.
			\item if $||u||\leq \rho$ and $\lambda=0$, then $\varphi(u)= \varphi(y)$ with $y \in Y$ and 
			$$\varphi(y) =- \frac{1}{2}||y||^{2}- \int_{\mathbb{R}^{N}} F(x,y) \leq 0,$$
		\end{itemize}
	because $F(x,t) \geq 0$.  In addition, $0 \in \mathcal{M}_{0}$  and $\varphi(0)=0$, then
		$$\max_{\mathcal{M}_{0}} \varphi=0.$$
		By $(f_{2})$, $\varphi$ maps bounded sets into bounded sets, hence $d=\sup_{\mathcal{M}} \varphi < \infty$.
	\end{proof}
	
The next lemma follows as in \cite[Lemma 6.15]{MW} and we omit its proof. 
\begin{lemma} \label{tau} The energy functional $\varphi$ is $\tau$-upper semicontinuous. 
	
\end{lemma}
%\begin{proof}
%In order to prove that 	$\varphi$ is $\tau$-upper semicontinuous, it is enough to show that $\varphi^{-1}((-\infty, a))$ is $\tau$-open for each $a \in \mathbb{R}$, or equivalently, that $\varphi^{-1}([a, +\infty))$ is $\tau$-closed for each $a \in \mathbb{R}$. Let $(u_n) \subset \varphi^{-1}([a, +\infty))$ be a sequence such that $u_{n} \stackrel{\tau}{\rightarrow} u$ in $H^{1}(\mathbb{R}^{N})$ for some $u \in H^{1}(\mathbb{R}^{N})$. We claim that $(u_{n})$ is bounded in $H^{1}(\mathbb{R}^{N})$. In fact, firstly  we know $u_{n}^{+}\rightarrow u^{+}$ in $Z$, then there is $M>0$ such that $||u_{n}^{+}|| \leq M, \;\forall\; n \in \mathbb{N}$. From this, 
%\begin{eqnarray}
%||u_{n}^{-}||^{2} &=& ||u_{n}^{+}||^{2} - 2\varphi(u_{n})- 2\int_{\mathbb{R}^{N}}F(x, u_{n}) \leq M^{2}- 2a, \nonumber 
%\end{eqnarray} 
%showing that $(u_{n})$ is bounded. From Lemma \ref{36}, $u_{n}^{-}\rightharpoonup u^{-}$ in $H^{1}(\mathbb{R}^{N})$, and going to a subsequence if necessary, $u_{n}(x)\rightarrow u(x)$ a.e in $\mathbb{R}^{N}$. The limits above together with the Fatou's lemma lead to 
%\begin{eqnarray}
%- \varphi(u)&=& \frac{1}{2}||u^{-}||^{2}-\frac{1}{2}||u^{+}||^{2}+ \int_{\mathbb{R}^{N}} F(x,u) \nonumber \\
%&\leq & \liminf_{n}\left[\frac{1}{2}||u_{n}^{-}||^{2}-\frac{1}{2}||u_{n}^{+}||^{2}+ \int_{\mathbb{R}^{N}} F(x,u_{n})\right] \nonumber \\
%&=& \liminf_{n}(-\varphi(u_{n})) \leq -a, \nonumber
%\end{eqnarray}
%that is, $u \in [\varphi^{-1}((- \infty, a))]^{c}$. This finishes the proof.
%\end{proof} 	

The Lemmas \ref{55}, \ref{60} and \ref{tau} yield the functional $\varphi$ satisfies the assumptions of Theorem \ref{59}, then the corollary below is true.  

	\begin{corollary}\label{68}
		Assume $(f_{1})-(f_{4})$. Then, there is $c \in [b,d]$ and a sequence $(u_{n}) \subset H^{1}(\mathbb{R}^{N})$ such that
		\begin{equation*}
		\varphi(u_{n})\rightarrow c\;\;\mbox{and}\;\; \lambda_{\varphi}(u_{n})\rightarrow 0.
		\end{equation*}
	\end{corollary}

The next lemma is crucial in our study, because it establishes the boundedness of $(PS)$ sequences of $\varphi$.

	\begin{lemma}\label{69}
		Suppose $(f_{1})-(f_{4})$. Then, every $(PS)$-sequence is bounded.
	\end{lemma}
	\begin{proof}
		
		Let $(u_{n}) \subset H^{1}(\mathbb{R}^{N})$ be a $(PS)_d$ sequence for $\varphi$, that is, 
		\begin{equation*}
		\varphi(u_{n})\rightarrow d\;\;\mbox{and}\;\; \lambda_{\varphi}(u_{n})\rightarrow 0,
		\end{equation*}
		for some $d \in \mathbb{R}$. Consider $(w_{n}) \subset \partial \varphi (u_{n})$ with $\lambda_{\varphi}(u_{n})= ||w_{n}||_{*}$ and 
		$$w_{n}= Q'(u_{n})- \rho_{n},$$
		where $\rho_{n} \in \partial \Psi(u_{n}) \subset L^{\tilde{\Phi}}(\mathbb{R}^{N})$. From $(f_{4})$,
		$$
		0<\theta F(x,t) \leq t \xi, \quad \forall  \xi \in \partial_tF(x,t) \quad \mbox{and} \quad \forall t \in \mathbb{R}, 
		$$
		and so, 
		\begin{equation*}\label{62}
		\frac{1}{2} \int_{\mathbb{R}^{N}} \rho_{n} u_{n}- \int_{\mathbb{R}^{N}} F(x, u_{n}) \geq \left(\frac{1}{2}- \frac{1}{\theta}\right)\int_{\mathbb{R}^{N}} \rho_{n} u_{n},
		\end{equation*}
	leading to 
		\begin{eqnarray}
		\varphi(u_{n})- \frac{1}{2} \left<w_{n}+ \rho_{n}, u_{n} \right> &=& - \int_{\mathbb{R}^{N}} F(x, u_{n}) \nonumber \\
		&=& - \int_{\mathbb{R}^{N}} F(x, u_{n}) + \frac{1}{2} \int_{\mathbb{R}^{N}} \rho_{n} u_{n} - \frac{1}{2} \int_{\mathbb{R}^{N}} \rho_{n} u_{n} \nonumber \\
		&=& \left[\frac{1}{2} \int_{\mathbb{R}^{N}} \rho_{n} u_{n} - \int_{\mathbb{R}^{N}} F(x, u_{n})\right]-\frac{1}{2} \int_{\mathbb{R}^{N}} \rho_{n} u_{n} \nonumber \\
		&\geq& \left(\frac{1}{2}- \frac{1}{\theta}\right)\int_{\mathbb{R}^{N}} \rho_{n} u_{n}-\frac{1}{2} \int_{\mathbb{R}^{N}} \rho_{n} u_{n}, \nonumber
		\end{eqnarray}
		that is,
		\begin{equation*}
		\varphi(u_{n})- \frac{1}{2} \left<w_{n}+ \rho_{n}, u_{n} \right> \geq \left(\frac{1}{2}- \frac{1}{\theta}\right)\int_{\mathbb{R}^{N}} \rho_{n} u_{n}-\frac{1}{2} \int_{\mathbb{R}^{N}} \rho_{n} u_{n}, 
		\end{equation*}
		or equivalently,
		\begin{equation}\label{63}
		\varphi(u_{n})- \frac{1}{2} \left<w_{n}, u_{n} \right> \geq \left(\frac{1}{2}- \frac{1}{\theta}\right)\int_{\mathbb{R}^{N}} \rho_{n} u_{n}. 
		\end{equation}
		
		Using the fact that $(\varphi(u_{n}))$ is bounded, there exists $M>0$ such that
		$$M- \frac{1}{2} \left<w_{n}, u_{n} \right> \geq \left(\frac{1}{2}- \frac{1}{\theta}\right)\int_{\mathbb{R}^{N}} \rho_{n} u_{n}, \;\forall\; n \in \mathbb{N}.$$
Since $||w_{n}||_{*}=o_{n}(1)$,  there exists $n_{0} \in \mathbb{N}$ such that
		\begin{equation}\label{64}
		(M+||u_{n}||)K \geq \int_{\mathbb{R}^{N}} \rho_{n} u_{n}\;\;\mbox{for}\;\; n \geq n_{0},
		\end{equation}
		where $K=\frac{2\theta}{\theta-2}>0$.
		
		Im the sequel we consider the set
		$$A_n= \left\{x \in \mathbb{R}^{N}\;:\; |u_{n}(x)| \leq 1\right\}.$$
		By (\ref{64}) and $(f_4)$,
		\begin{equation} \label{65}
		\int_{A_n} \rho_{n} u_{n} \leq (M+||u_{n}||)K\;\;\mbox{for}\;\; n \geq n_{0} 
		\end{equation}
		and
		\begin{equation} \label{66}
		\int_{\mathbb{R}^{N}\backslash A_n} \rho_{n} u_{n} \leq (M+||u_{n}||)K\;\;\mbox{for}\;\; n \geq n_{0}.
		\end{equation}
	
		\begin{claim}\label{61}
	Assume $(f_{1})-(f_{4})$ and let $r= \frac{p}{p-1}$ be the conjugate exponent of $p>1$. Then,  there is $\tilde{c}>0$ such that
		\begin{itemize}
			\item [(i)] $|\rho|^{2} \leq \tilde{c}\; t \rho$ for $|t|\leq 1$,
			\item [(ii)] $|\rho|^{r} \leq \tilde{c} \;t \rho$ for $|t|\geq 1$,
		\end{itemize}
for all $\rho \in \partial_{t}F(x,t)$.	
\end{claim}
Indeed, for $|t|\leq 1$, by $(F_{*})$ 
		$$|\rho| \leq C_{0}\;(|t|+|t|^{p-1}) \leq 2\;C_{0}|t|,$$
		that is,
		$$|\rho|^{2} \leq  2\;C_{0}|t||\rho|=2C_{0}\;t\rho.$$
		For $|t|\geq 1$, again by $(F_{*})$, we get
		\begin{eqnarray}
		|\rho|^{r-1} &\leq& C_{0}^{r-1} (|t|+|t|^{p-1})^{r-1} \nonumber \\
		&\leq & C_{0}^{r-1} (1+2|t|^{p-1})^{r-1} \nonumber \\
		&\leq & C_{0}^{r-1} k_{1}+ C_{0}^{p-1} k_{2} |t|^{(r-1)(p-1)} \nonumber \\
		&\leq & C_{1} |t|+ C_{2}|t|\leq  \tilde{C} |t|, \nonumber
		\end{eqnarray}
		that is,
		$$|\rho|^{r} \leq \tilde{C} \; |t||\rho|= \tilde{C} \;t\rho.$$
		Now the claim follows with $\tilde{c}= \max\{\tilde{C}, 2C\}>0$.
		
		Denote
		$$
		y_{n}= u_{n}^{-}\;\;\mbox{and}\;\; z_{n}=u_{n}^{+}.
		$$
		The H\"older inequality together with Claim \ref{61}, Sobolev embedding and (\ref{65}) establish 
		\begin{eqnarray}
		||y_{n}||^{2}&=& - \left<w_{n}- \rho_{n}, y_{n}\right> \nonumber \\
		&= & - \left<w_{n}, y_{n}\right>- \left<\rho_{n}, y_{n}\right> \nonumber \\ 
		&\leq & |\left<w_{n}, y_{n}\right>|+ |\left<\rho_{n}, y_{n}\right>| \nonumber \\
		&\leq & ||y_{n}||+ \int_{\mathbb{R}^{N}} |\rho_{n}|\; |y_{n}|  \nonumber \\
		&\leq & ||y_{n}||+\left(\int_{A_n} |\rho_{n}|^{2}\right)^{\frac{1}{2}}\left(\int_{A_n} |y_{n}|^{2}\right)^{\frac{1}{2}}+ \left(\int_{\mathbb{R}^{N}\backslash A_n} |\rho_{n}|^{r}\right)^{\frac{1}{r}}\left(\int_{\mathbb{R}^{N}\backslash A_n} |y_{n}|^{p}\right)^{\frac{1}{p}} \nonumber \\
		&\leq&  ||y_{n}||+\left(\tilde{c}\int_{A_n} \rho_{n} u_{n}\right)^{\frac{1}{2}}C_{1}||y_{n}||+ \left(\tilde{c}\int_{\mathbb{R}^{N}\backslash A_n}  \rho_{n} u_{n}\right)^{\frac{1}{r}}C_{2}||y_{n}|| \nonumber \\
		&\leq& ||y_{n}||+ \tilde{c}^{\frac{1}{2}}[K(M+||u_{n}||)]^{\frac{1}{2}}C_{1}||y_{n}||+ \tilde{c}^{\frac{1}{r}}[K(M+||u_{n}||)]^{\frac{1}{r}}C_{2}||y_{n}||, \nonumber
		\end{eqnarray}
		that is,
		$$||y_{n}||^{2} \leq ||y_{n}||+ \tilde{c}^{\frac{1}{2}}[K(M+||u_{n}||)]^{\frac{1}{2}}C_{1}||y_{n}||+ \tilde{c}^{\frac{1}{r}}[K(M+||u_{n}||)]^{\frac{1}{r}}C_{2}||y_{n}||,\;n \geq n_{0}.$$
		Analogously, by (\ref{66}),
		$$||z_{n}||^{2} \leq ||z_{n}||+ \tilde{c}^{\frac{1}{2}}[K(M+||u_{n}||)]^{\frac{1}{2}}C_{3}||z_{n}||+ \tilde{c}^{\frac{1}{r}}[K(M+||u_{n}||)]^{\frac{1}{r}}C_{4}||z_{n}||,\;n \geq n_{0}.$$
		Therefore, for $n \geq n_{0}$
		\begin{eqnarray}
		||u_{n}||^{2} &\leq& (||y_{n}||+||z_{n}||)[K(M+||u_{n}||)]^{\frac{1}{2}}M_{1}+ (||y_{n}||+||z_{n}||)[K(M+||u_{n}||)]^{\frac{1}{r}} M_{2} \nonumber \\
		&\leq & 2||u_{n}||[K(M+||u_{n}||)]^{\frac{1}{2}}M_{1}+ 2||u_{n}||[K(M+||u_{n}||)]^{\frac{1}{r}} M_{2}. \nonumber
		\end{eqnarray}
The aforementioned inequality yields $(u_n)$ is bounded. 
		
	\end{proof}
	
	Now, we are ready to conclude the proof of Theorem \ref{Teorema1}. \\

	\noindent {\bf Proof of Theorem \ref{Teorema1}:} \,\, 
	
\begin{proof}
		By Corollary \ref{68}  and Lemma \ref{69}, there exists a bounded sequence $(u_{n}) \subset H^{1}(\mathbb{R}^{N})$ satisfying 
		\begin{equation*}
		\varphi(u_{n})\rightarrow c>0\;\;\mbox{and}\;\; \lambda_{\varphi}(u_{n})\rightarrow 0.
		\end{equation*}
	
		\begin{claim}
			There exists $\delta>0$ such that
			$$\liminf_{n} \sup_{y \in \mathbb{R}^{N}} \int_{B(y,1)} |u_{n}|^{2}\geq \delta.$$
		\end{claim}
	If the claim is not true, we must have 
		$$\liminf_{n} \sup_{y \in \mathbb{R}^{N}} \int_{B(y,1)} |u_{n}|^{2}=0.$$
	Thus, by Lions \cite [Lemma 1.21]{MW},  $u_{n}\rightarrow 0$ in $L^{s}(\mathbb{R}^{N})$ for $2<s<2^{*}$.
		On the other hand,
		\begin{equation}\label{70}
		0<c=\varphi(u_{n})- \frac{1}{2}\left<w_{n}, u_{n} \right>+o_{n}(1)= \int_{\mathbb{R}^{N}} \rho_{n} u_{n} +o_{n}(1),
		\end{equation}
		where $w_{n}= Q'(u_{n})- \rho_{n}$ with $\lambda_{\varphi}(u_{n})=||w_{n}||_{*}$ and $\rho_{n} \in  \partial \Psi(u_{n})$. 
		
		For any $\varepsilon>0$, by $(f_{1})-(f_{3})$, we get
		$$\int_{\mathbb{R}^{N}} \rho_{n} u_{n} \leq \frac{\varepsilon}{2}||u_{n}||_2^{2}+c_{2}||u_{n}||_p^{p}.$$
		Using the fact that $(u_{n})$ is bounded, Sobolev embedding  and $u_{n}\rightarrow 0$ in $L^{p}(\mathbb{R}^{N})$ 
		$$\int_{\mathbb{R}^{N}} \rho_{n} u_{n}\rightarrow 0$$
	contrary to (\ref{70}).
		
	From this, going to a subsequence if necessary, there exists $n_{0} \in \mathbb{N}$ such that
		$$\sup_{y \in \mathbb{R}^{N}} \int_{B(y,1)} |u_{n}|^{2} \geq \frac{\delta}{2}, \;n \geq n_{0}.$$
		By definition of supreme, there exists $(y_{n}) \subset \mathbb{R}^{N}$ such that
		\begin{equation*}
		\int_{B(y_{n},1)} |u_{n}|^{2} \geq \frac{\delta}{4}, \;n \geq n_{0}.
		\end{equation*}
Then, there exists $(z_{n}) \subset \mathbb{Z}^{N}$ such that
			\begin{equation*}
			\int_{B(z_{n},1+\sqrt{N})} |u_{n}|^{2} \geq \frac{\delta}{4}, \;n \geq n_{0}.
			\end{equation*}
	Setting $v_{n}(x)=u_{n}(x+z_{n})$, we compute 
		\begin{equation}
		\int_{B(0,1+\sqrt{N})} |v_{n}(x)|^{2}= \int_{B(z_{n},1+\sqrt{N})} |u_{n}(x)|^{2} \geq \frac{\delta}{4},\;n \geq n_{0}.
		\end{equation}
	As $(u_n)$ is a bounded sequence, it follows that $(v_n)$ is also a bounded sequence.  Hence, supposing that for some subsequence $v_n \rightharpoonup v$ in $H^{1}(\mathbb{R}^N)$, we get 
	$$
	\int_{B(0,1+\sqrt{N})}|v|^{2}=\lim_{n \to +\infty}\int_{B(0,1+\sqrt{N})}|v_{n}(x)|^{2} \geq \frac{\delta}{4}>0,
	$$ 
	showing that $v \not= 0$.

		% By analogous arguments used in \cite{Rubia} Prposition 3.1, we get $(v_{n}) \subset H^{1}(\mathbb{R}^{N})$ is a  $(PS)_{c}$-sequence.
		\begin{claim}
			$(v_{n}) \subset H^{1}(\mathbb{R}^{N})$ is also a $(PS)_{c}$ sequence for $\varphi$.
		\end{claim}
		
		By change variable, it is immediate to see that 
		$$
		\varphi(u_{n})= \varphi(v_{n}),\quad \forall n \in \mathbb{N},
		$$
		and so,
		$$\varphi(v_{n})\rightarrow c.$$
		Now, we will show tha $\lambda_{\varphi}(v_{n})\rightarrow 0$ 
		whenever $n\rightarrow +\infty$. 
		
		First of all, we recall that there is $(w_{n}) \subset (H^{1}(\mathbb{R}^{N}))^{*}$ satisfying 
		$$
		w_{n}= Q'(u_{n})- \rho_{n}, \;\lambda_{\varphi}(u_{n})=||w_{n}||_{*},
		$$
		with $(\rho_{n}) \subset \partial \Psi(u_{n})$. Thus,   
		$$\Psi^{\circ}(u_{n}; \phi) \geq \left<\rho_{n}, \phi \right>,\forall\; \phi \in L^{\Phi}(\mathbb{R}^{N}).$$
		Taking $\phi \in H^{1}(\mathbb{R}^{N})$, we get
		\begin{eqnarray}
		\Psi^{\circ}(u_{n}; \phi(\cdot-z_{n})) &\geq&  \left<\rho_{n}, \phi(\cdot- z_{n}) \right> \nonumber \\
		&= &\int_{\mathbb{R}^{N}} \rho_{n} \phi(\cdot-z_{n}) \nonumber \\
		&=&\int_{\mathbb{R}^{N}} \rho_{n}(\cdot +z_{n}) \phi \nonumber \\
		&=& \left<\tilde{\rho}_{n}, \phi \right>, \label{71}
		\end{eqnarray}
		where $\tilde{\rho}_{n}(x)= \rho_{n}(x+z_{n}),\;\forall\; x \in \mathbb{R}^{N}$. Recalling that
		\begin{equation*}
		\Psi(u_{n}+h+\lambda \phi(\cdot-z_{n}))= \int_{\mathbb{R}^{N}} F(x, u_{n}+h+\lambda \phi(\cdot-z_{n}))
		\end{equation*}
		and
		\begin{equation*}
		\Psi(u_{n}+h)= \int_{\mathbb{R}^{N}} F(x, u_{n}+h),
		\end{equation*}
		we find
		\begin{eqnarray}
		\Psi(u_{n}+h+\lambda \phi(\cdot-z_{n}))&=&\int_{\mathbb{R}^{N}}F(x, u_{n}(x)+h(x)+\lambda \phi(x-z_{n})) \nonumber \\
		&=& \int_{\mathbb{R}^{N}}F(x, u_{n}(x+z_{n})+h(x+z_{n})+\lambda \phi(x)) \nonumber \\
		&=& \int_{\mathbb{R}^{N}}F(x, v_{n}(x)+\tilde{h}(x)+\lambda \phi(x)) \nonumber \\
		&=& \Psi(v_{n}+\tilde{h}+\lambda \phi), \nonumber
		\end{eqnarray}
		where $\tilde{h}(x)=h(x+z_{n}), \;\forall\; x \in \mathbb{R}^{N}$. An analogous argument shows $\Psi(u_{n}+h)=\Psi(v_{n}+ \tilde{h})$. Hence, 
		$$\Psi^{\circ}(u_{n}; \phi(\cdot-z_{n}))= \Psi^{\circ}(v_{n}, \phi).$$
		By (\ref{71}), 
		$$ \Psi^{\circ}(v_{n}, \phi) \geq \left<\tilde{\rho}_{n}, \phi \right>, \;\forall\; \phi \in H^{1}(\mathbb{R}^{N}),$$
		then $\tilde{\rho}_ {n} \in \partial \Psi(v_{n})$. Since $w_{n}=Q'(u_{n})- \rho_{n}$, we also have 
		\begin{eqnarray}
		\left<w_{n}, \phi(\cdot-z_{n}) \right>&=& \left<Q'(u_{n}), \phi(\cdot-z_{n}) \right>- \left<\rho_{n}, \phi(\cdot-z_{n}) \right> \nonumber \\
		&=& \left<Q'(v_{n}), \phi \right>- \left<\tilde{\rho}_{n}, \phi \right>. \nonumber
		\end{eqnarray}
		Setting $\left<\tilde{w}_{n}, \phi \right>=\left<w_{n}, \phi(\cdot-z_{n}) \right>$, we assert
		$$\tilde{w}_{n}= Q'(v_{n})- \tilde{\rho}_{n}.$$
		\begin{claim}
			$\tilde{w}_{n} \in \partial \varphi(v_{n})$.
		\end{claim}
	As $w_{n} \in \partial \varphi(u_{n})$, then
		$$ \left<w_{n}, \phi \right> \leq \varphi^{\circ}(u_{n}; \phi),\; \forall\; \phi \in H^{1}(\mathbb{R}^{N}).$$
		Thereby
		\begin{eqnarray}
		\varphi^{\circ}(u_{n}; \phi(\cdot-z_{n})) &\geq& \left<w_{n}, \phi(\cdot-z_{n}) \right> \nonumber \\
		&=& \left<\tilde{w}_{n}, \phi \right>,\;\forall\; \phi \in H^{1}(\mathbb{R}^{N}). \nonumber
		\end{eqnarray}
		On the other hand, a simple change variable implies 
		$$\varphi^{\circ}(u_{n}; \phi(\cdot-z_{n}))=\varphi^{\circ}(v_{n}; \phi),$$
		then
		$$\left<\tilde{w}_{n}, \phi \right> \leq \varphi^{\circ}(v_{n}; \phi) ,\;\forall\; \phi \in H^{1}(\mathbb{R}^{N}),$$
	proving the claim. 	
	
	Now, by definition of $\tilde{w}_{n}$, it is easy do check that
		$$||\tilde{w}_{n}||_{*} \leq ||w_{n}||_{*},\;\forall\; n \in \mathbb{N}.$$ 
		Therefore,
		$$0 \leq \lambda_{\varphi}(v_{n}) \leq ||\tilde{w}_{n}||_{*} \leq ||w_{n}||_{*}\rightarrow 0,$$ 
		that is,
		$$\lambda_{\varphi}(v_{n})\rightarrow 0,\;\;\mbox{as}\;\; n\rightarrow +\infty.$$
	Now our goal is to prove that 
	$$-\Delta v(x) + V(x) v(x) \in [\underline{f}(x,v(x)), \overline{f}(x,v(x))]\;\;\mbox{a.e in}\;\; \mathbb{R}^{N},$$
	where $v$ is the weak limit of $(v_n)$ in $H^{1}(\mathbb{R}^N)$. 
	
	From the study above, there exists $(\tilde{\omega}_{n}) \subset \partial \varphi(v_{n})$ such that $\tilde{\omega}_{n}= Q'(v_{n})- \tilde{\rho}_{n}$ and $||\tilde{\omega}_{n}||_{*}=o_{n}(1)$ where $(\tilde{\rho}_{n}) \subset \partial \Psi(v_{n})$. For $\phi \in H^{1}(\mathbb{R}^{N})$, we obtain
	\begin{eqnarray}
	\left <\tilde{\rho}_{n}, \phi \right>= \left <Q'(v_{n}), \phi \right>-\left <\tilde{\omega}_{n}, \phi \right>\rightarrow \left <Q'(v), \phi \right>, \;\mbox{as}\;\;n\rightarrow +\infty,\nonumber
	\end{eqnarray}
	that is, $\tilde{\rho}_{n} \stackrel{*}{\rightharpoonup} Q'(v)$ in $(H^{1}(\mathbb{R}^{N}))^{*}$. Then, by Proposition \ref{7},  $Q'(v) \in \partial \Psi (v)$. Thereby, $Q'(v)= \rho \in \partial \Psi (v)$, and so, 
	$$\int_{\mathbb{R}^{N}}(\nabla v \nabla \phi + V v \phi)= \int_{\mathbb{R}^{N}} \rho \phi\;\;\mbox{for all}\;\;\phi \in H^{1}(\mathbb{R}^{N}),$$
	where $\rho(x) \in [\underline{f}(x,v(x)), \overline{f}(x,v(x))]\;\;\mbox{a.e in}\;\; \mathbb{R}^{N}.$ Hence
	$$
	\left\{\begin{aligned}
	-\Delta v + V(x)v &= \rho(x)\;\;\mbox{in}\;\;\mathbb{R}^{N}, \\
	v \in H^{1}(\mathbb{R}^{N}).
	\end{aligned}
	\right.
	$$
Since $\rho \in L_{loc}^{\frac{p}{p-1}}(\mathbb{R}^{N})$, the elliptic regularity theory gives that $v \in W_{loc}^{2,\frac{p}{p-1}}(\mathbb{R}^{N})$ and
	$$-\Delta v + V(x)v = \rho(x)\;\;\mbox{a.e in}\;\; \mathbb{R}^{N},$$
that is,
$$
-\Delta v(x) + V(x) v(x) \in [\underline{f}(x,v(x)), \overline{f}(x,v(x))]\;\;\mbox{a.e in}\;\; \mathbb{R}^{N},
$$
finishing the proof. \end{proof}

	%%%%%%%%%%%%%%%%%%%%%%%%%%%%%%%%%%%%
	
\end{document}